\documentclass[11pt,a4paper]{article}

\usepackage{pgf,tikz,tkz-graph}
\usepackage{physics}
\usepackage{listings}
\usepackage[english]{babel}

\usepackage{mathrsfs}
\usetikzlibrary{arrows}
\usetikzlibrary{arrows.meta}
\usepackage[T1]{fontenc}
\usepackage{amsfonts}
\usepackage{amssymb}
\usepackage{amsthm}
\usepackage{amscd}
\usepackage{graphicx}
\usepackage{enumitem}
\usepackage{authblk}
\usepackage{verbatim}
\usepackage{hyperref}
\usepackage{amsmath,caption}
\usepackage{url,pdfpages,xcolor,framed,color}
\usepackage{a4wide}

\theoremstyle{plain}

\newtheorem{thr}{Theorem}[section]

\newtheorem{p}[thr]{Problem}
\newtheorem{lem}[thr]{Lemma}
\newtheorem{prop}[thr]{Proposition}
\newtheorem{conj}[thr]{Conjecture}
\theoremstyle{definition}
\newtheorem{defi}[thr]{Definition}
\newtheorem{cor}[thr]{Corollary}

\def\P{\mathcal{P}}
\def\Q{\mathcal{Q}}

\DeclareMathOperator{\ecc}{ecc}
\DeclareMathOperator{\rad}{rad}

\title{Extremal total distance of graphs of given radius}
\author{Stijn Cambie\footnote{Department of Mathematics, Radboud University Nijmegen, Postbus 9010, 6500 GL Nijmegen, The Netherlands. Email: \href{mailto:S.Cambie@math.ru.nl}{S.Cambie@math.ru.nl}. This work has been supported by a Vidi Grant of the Netherlands Organization for Scientific Research (NWO), grant number $639.032.614$.} }%
\date{}

\begin{document}
	\definecolor{xdxdff}{rgb}{0.49019607843137253,0.49019607843137253,1.}
	\definecolor{ududff}{rgb}{0.30196078431372547,0.30196078431372547,1.}
	
	\tikzstyle{every node}=[circle, draw, fill=black!50,
	inner sep=0pt, minimum width=4pt]

	\maketitle

	\begin{abstract}
		In 1984, Plesn\'{i}k determined the minimum total distance for given order and diameter and characterized the extremal graphs and digraphs.
		We prove the analog for given order and radius, when the order is sufficiently large compared to the radius.
		This confirms asymptotically a conjecture of Chen et al. 
		We show the connection between minimizing the total distance and maximizing the size under the same conditions. 
		We also prove some asymptotically optimal bounds for the maximum total distance.
	\end{abstract}

	\section{Introduction}

The total distance $W(G)$ of a graph $G$ equals the sum of distances between all unordered pairs of vertices, i.e. $W(G)=\sum_{\{u,v\} \subset V} d(u,v).$
In 1984, Plesn\'{i}k~\cite{P84} determined the minimum total distance among all graphs of order $n$ and diameter $d$. He did this both for graphs and digraphs and characterized the extremal examples.
In this paper we solve the analogous questions for given order $n$ and radius $r$, when $n$ is sufficiently large compared with $r$.

The extremal graphs attaining the minimal total distance among all graphs with order $n$ and radius $r \in \{1,2\}$ are easily characterized; complete graphs when $r=1$, complete graphs minus a maximum matching when $r=2$ and $2 \mid n$ and complete graphs minus a maximum matching and an additional edge adjacent to the vertex not in the maximum matching, when $r=2$ and $2 \nmid n.$
For $r\ge 3$ the question is harder and a conjecture of the extremal graphs was made by Chen, Wu and An~\cite{Chen}. 
Here $G_{n,r,s}$ is a cycle $C_{2r}$ in which we take blow-ups in $2$ consecutive vertices by cliques $K_s$ and $K_{n-2r+2-s}$, as defined in Section~\ref{not&def}.
\begin{conj}[\cite{Chen}]\label{conjchen}
	Let $n$ and $r$ be two positive integers with $n \ge 2r$ and $r \ge 3$. For any graph $G$ of order $n$ with radius $r$, $W(G) \ge W(G_{n,r,1})$. Equality holds if and only if $G \cong G_{n,r,s}$ for some $1 \le s \le \frac{n-2r+2}{2}$.
\end{conj}

Although there are counterexamples to Conjecture~\ref{conjchen} when $n$ is small (as we discuss below), we will show that Conjecture~\ref{conjchen} is true asymptotically, i.e. when $n\ge n_1(r)$ for some value $n_1(r)$, in Section~\ref{AsProofChen}.
\begin{thr}\label{main}
	For any $r \ge 3$, there exists a value $n_1(r)$ such that for all $n \ge n_1(r)$ the following hold
	\begin{itemize}
		\item any graph $G$ of order $n$ with radius $r$ satisfies $W(G) \ge W(G_{n,r,1})$. Equality holds if and only if $G \cong G_{n,r,s}$ for some $1 \le s \le \frac{n-2r+2}{2}$.
	\end{itemize}
	
\end{thr}

There are some main ideas in the proof which are somewhat intuitive.
Since the minimum average distance is close to $1$, we expect there are many vertices of high degree and there is a large clique.
Because the conjectured extremal graphs contain large blow-ups, we can expect there are vertices such that $G \backslash v$ satisfies the original statement as well. By proving that the total distance $W(G)$ and $W(G \backslash v)$ differ by a certain amount with equality if and only if a structure close to the conjectured structure appears, at the end we only need to prove that an extremal graph is exactly of the form $G_{n,r,s}.$

For small values of $n$ with respect to a fixed $r$, there might be a few exceptions to Conjecture~\ref{conjchen}. 
The graph $Q_3$ is a counterexample for the equality statement when $r=3$ and $n=8$, as it also has a total distance equal to $48.$ 
A computer check has shown that this is the only counterexample for $n<10.$
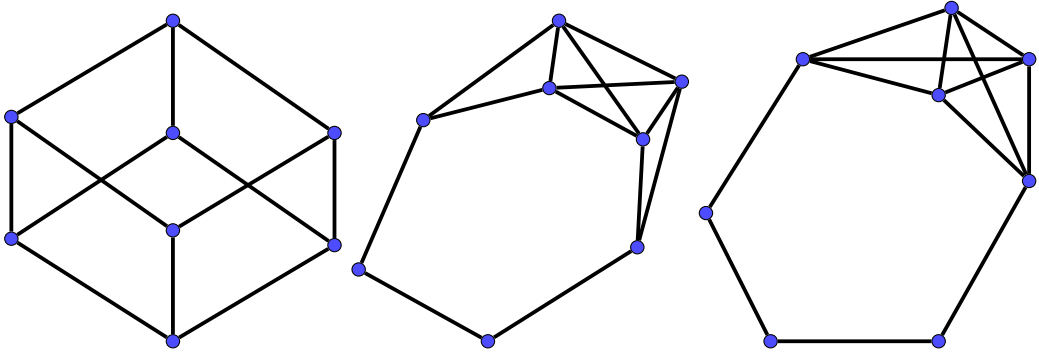
\begin{figure}[h]

	\begin{tikzpicture}[scale=0.85]
	\definecolor{cv0}{rgb}{0.0,0.0,0.0}
	\definecolor{cfv0}{rgb}{1.0,1.0,1.0}
	\definecolor{clv0}{rgb}{0.0,0.0,0.0}
	\definecolor{cv1}{rgb}{0.0,0.0,0.0}
	\definecolor{cfv1}{rgb}{1.0,1.0,1.0}
	\definecolor{clv1}{rgb}{0.0,0.0,0.0}
	\definecolor{cv2}{rgb}{0.0,0.0,0.0}
	\definecolor{cfv2}{rgb}{1.0,1.0,1.0}
	\definecolor{clv2}{rgb}{0.0,0.0,0.0}
	\definecolor{cv3}{rgb}{0.0,0.0,0.0}
	\definecolor{cfv3}{rgb}{1.0,1.0,1.0}
	\definecolor{clv3}{rgb}{0.0,0.0,0.0}
	\definecolor{cv4}{rgb}{0.0,0.0,0.0}
	\definecolor{cfv4}{rgb}{1.0,1.0,1.0}
	\definecolor{clv4}{rgb}{0.0,0.0,0.0}
	\definecolor{cv5}{rgb}{0.0,0.0,0.0}
	\definecolor{cfv5}{rgb}{1.0,1.0,1.0}
	\definecolor{clv5}{rgb}{0.0,0.0,0.0}
	\definecolor{cv6}{rgb}{0.0,0.0,0.0}
	\definecolor{cfv6}{rgb}{1.0,1.0,1.0}
	\definecolor{clv6}{rgb}{0.0,0.0,0.0}
	\definecolor{cv7}{rgb}{0.0,0.0,0.0}
	\definecolor{cfv7}{rgb}{1.0,1.0,1.0}
	\definecolor{clv7}{rgb}{0.0,0.0,0.0}
	\definecolor{cv0v4}{rgb}{0.0,0.0,0.0}
	\definecolor{cv0v5}{rgb}{0.0,0.0,0.0}
	\definecolor{cv0v6}{rgb}{0.0,0.0,0.0}
	\definecolor{cv1v4}{rgb}{0.0,0.0,0.0}
	\definecolor{cv1v5}{rgb}{0.0,0.0,0.0}
	\definecolor{cv1v7}{rgb}{0.0,0.0,0.0}
	\definecolor{cv2v4}{rgb}{0.0,0.0,0.0}
	\definecolor{cv2v6}{rgb}{0.0,0.0,0.0}
	\definecolor{cv2v7}{rgb}{0.0,0.0,0.0}
	\definecolor{cv3v5}{rgb}{0.0,0.0,0.0}
	\definecolor{cv3v6}{rgb}{0.0,0.0,0.0}
	\definecolor{cv3v7}{rgb}{0.0,0.0,0.0}

	\node[label=center:\(\),fill=ududff,style={minimum size=5,shape=circle}] (v0) at (2.5,1.73) {};
	\node[label=center:\(\),fill=ududff,style={minimum size=5}] (v1) at (2.5,5) {};
	\node[label=center:\(\),fill=ududff,style={minimum size=5,shape=circle}] (v2) at (5,1.5) {};
	\node[label=center:\(\),fill=ududff,style={minimum size=5,shape=circle}] (v3) at (0,1.6) {};
	\node[label=center:\(\),fill=ududff,style={minimum size=5,shape=circle}] (v4) at (5,3.25) {};
	\node[label=center:\(\),fill=ududff,style={minimum size=5,shape=circle}] (v5) at (0,3.5) {};
	\node[label=center:\(\),fill=ududff,style={minimum size=5,shape=circle}] (v6) at (2.5,0) {};
	\node[label=center:\(\),fill=ududff,style={minimum size=5,shape=circle}] (v7) at (2.5,3.25) {};

	%
\Edge[lw=0.05cm,style={color=cv0v4,},](v0)(v4)
\Edge[lw=0.05cm,style={color=cv0v5,},](v0)(v5)
\Edge[lw=0.05cm,style={color=cv0v6,},](v0)(v6)
\Edge[lw=0.05cm,style={color=cv1v4,},](v1)(v4)
\Edge[lw=0.05cm,style={color=cv1v5,},](v1)(v5)
\Edge[lw=0.05cm,style={color=cv1v7,},](v1)(v7)
\Edge[lw=0.05cm,style={color=cv2v4,},](v2)(v4)
\Edge[lw=0.05cm,style={color=cv2v6,},](v2)(v6)
\Edge[lw=0.05cm,style={color=cv2v7,},](v2)(v7)
\Edge[lw=0.05cm,style={color=cv3v5,},](v3)(v5)
\Edge[lw=0.05cm,style={color=cv3v6,},](v3)(v6)
\Edge[lw=0.05cm,style={color=cv3v7,},](v3)(v7)
	\end{tikzpicture}
	\begin{tikzpicture}[scale=0.85]

	\node[label=center:\(\),fill=ududff,style={minimum size=5,shape=circle}] (v0) at (2.9511,3.9498) {};
	\node[label=center:\(\),fill=ududff,style={minimum size=5}] (v1) at (4.3115,1.4668) {};
	\node[label=center:\(\),fill=ududff,style={minimum size=5,shape=circle}] (v2) at (0,1.12) {};
	\node[label=center:\(\),fill=ududff,style={minimum size=5,shape=circle}] (v3) at (3.1,5) {};
	\node[label=center:\(\),fill=ududff,style={minimum size=5,shape=circle}] (v4) at (2,0) {};
	\node[label=center:\(\),fill=ududff,style={minimum size=5,shape=circle}] (v5) at (5,4.05) {};
	\node[label=center:\(\),fill=ududff,style={minimum size=5,shape=circle}] (v6) at (1,3.45) {};
	\node[label=center:\(\),fill=ududff,style={minimum size=5,shape=circle}] (v7) at (4.4,3.15) {};

	\Edge[lw=0.05cm](v0)(v3)
	\Edge[lw=0.05cm](v0)(v5)
	\Edge[lw=0.05cm,](v0)(v6)
	\Edge[lw=0.05cm](v0)(v7)
	\Edge[lw=0.05cm](v1)(v4)
	\Edge[lw=0.05cm](v1)(v5)
	\Edge[lw=0.05cm](v1)(v7)
	\Edge[lw=0.05cm](v2)(v4)
	\Edge[lw=0.05cm](v2)(v6)
	\Edge[lw=0.05cm](v3)(v5)
	\Edge[lw=0.05cm](v3)(v6)
	\Edge[lw=0.05cm](v3)(v7)
	\Edge[lw=0.05cm](v5)(v7)
	\end{tikzpicture}
	\begin{tikzpicture}[scale=0.85]

	\definecolor{cv0v3}{rgb}{0.0,0.0,0.0}
	\definecolor{cv0v5}{rgb}{0.0,0.0,0.0}
	\definecolor{cv0v6}{rgb}{0.0,0.0,0.0}
	\definecolor{cv0v7}{rgb}{0.0,0.0,0.0}
	\definecolor{cv1v4}{rgb}{0.0,0.0,0.0}
	\definecolor{cv1v5}{rgb}{0.0,0.0,0.0}
	\definecolor{cv2v4}{rgb}{0.0,0.0,0.0}
	\definecolor{cv2v6}{rgb}{0.0,0.0,0.0}
	\definecolor{cv3v5}{rgb}{0.0,0.0,0.0}
	\definecolor{cv3v6}{rgb}{0.0,0.0,0.0}
	\definecolor{cv3v7}{rgb}{0.0,0.0,0.0}
	\definecolor{cv5v7}{rgb}{0.0,0.0,0.0}
	\definecolor{cv6v7}{rgb}{0.0,0.0,0.0}
	%
	
	\node[label=center:\(\),fill=ududff,style={minimum size=5,shape=circle}] (v0) at (3.8,5.2) {};
	\node[label=center:\(\),fill=ududff,style={minimum size=5}] (v1) at (0,2) {};
	\node[label=center:\(\),fill=ududff,style={minimum size=5,shape=circle}] (v2) at (3.6,0) {};
	\node[label=center:\(\),fill=ududff,style={minimum size=5,shape=circle}] (v3) at (3.6,3.84) {};
	\node[label=center:\(\),fill=ududff,style={minimum size=5,shape=circle}] (v4) at (1,0) {};
	\node[label=center:\(\),fill=ududff,style={minimum size=5,shape=circle}] (v5) at (1.5,4.4) {};
	\node[label=center:\(\),fill=ududff,style={minimum size=5,shape=circle}] (v6) at (5,2.5) {};
	\node[label=center:\(\),fill=ududff,style={minimum size=5,shape=circle}] (v7) at (5,4.4) {};

	\Edge[lw=0.05cm,style={color=cv0v3,},](v0)(v3)
\Edge[lw=0.05cm,style={color=cv0v5,},](v0)(v5)
\Edge[lw=0.05cm,style={color=cv0v6,},](v0)(v6)
\Edge[lw=0.05cm,style={color=cv0v7,},](v0)(v7)
\Edge[lw=0.05cm,style={color=cv1v4,},](v1)(v4)
\Edge[lw=0.05cm,style={color=cv1v5,},](v1)(v5)
\Edge[lw=0.05cm,style={color=cv2v4,},](v2)(v4)
\Edge[lw=0.05cm,style={color=cv2v6,},](v2)(v6)
\Edge[lw=0.05cm,style={color=cv3v5,},](v3)(v5)
\Edge[lw=0.05cm,style={color=cv3v6,},](v3)(v6)
\Edge[lw=0.05cm,style={color=cv3v7,},](v3)(v7)
\Edge[lw=0.05cm,style={color=cv5v7,},](v5)(v7)
\Edge[lw=0.05cm,style={color=cv6v7,},](v6)(v7)

	\end{tikzpicture}

	\caption{The three extremal graphs for $r=3$ and $n=8$: $Q_3, G_{8,3,2}$ and $ G_{8,3,1}$ }
	\label{fig:extremaln8r3}
\end{figure}

In the digraph setting, the total distance $W(D)$ of a digraph $D$ equals the sum of all distances between all ordered pairs of vertices. The outradius of $D$ is equal to the smallest value $r$ such that there exists a vertex $x$ for which $d(x,v)\le r$ for every vertex $v$ of the digraph $D$.

In this setting, when the outradius $r=1$, the extremal digraphs are obviously bidirected cliques, their total distance being $2\binom{n}{2}.$
When $r=2$, the extremal digraphs are bidirected cliques missing $n$ edges, one starting in every vertex, with the restriction that no $n-1$ of those missing edges end in the same vertex. In this case the total distance $W(D)=n^2.$
For $r\ge 3$, we propose the digraph analog to Conjecture~\ref{conjchen}. The definition of $D_{n,r,s}$ being stated in Section~\ref{not&def}. Figure~\ref{fig:Dnrs} shows $D_{n,r,s}$ for $r=3.$

\begin{conj}\label{conjchendigraph}
	Let $n$ and $r$ be two positive integers with $n \ge 2r$ and $r \ge 3$. For any digraph $D$ or order $n$ with outradius $r$, $W(D) \ge W(D_{n,r,1})$. Equality holds if and only if $D \cong D_{n,r,s}$ for $1 \le s \le \frac{n-2r+2}{2}$.
\end{conj}

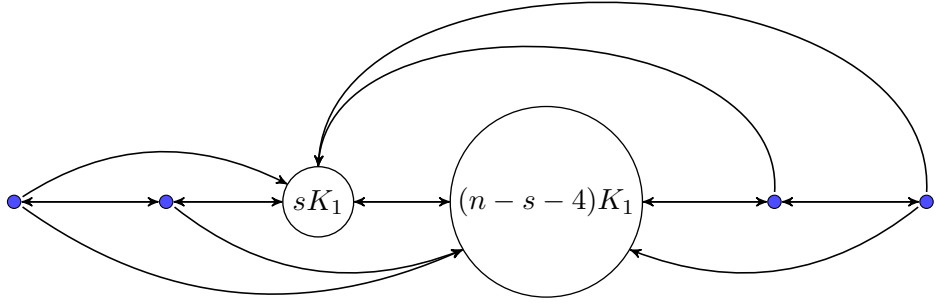
\begin{figure}[h]
	\centering

	\begin{tikzpicture}
	\definecolor{cv0}{rgb}{0.0,0.0,0.0}
	\definecolor{c}{rgb}{1.0,1.0,1.0}
	
	\node[label=north:\(\),fill=ududff,style={minimum size=5}] (v1) at (13,0) {};
	\node[label=north:\(\),fill=ududff,style={minimum size=5}] (v2) at (15,0) {};
	\node[label=north:\(\),fill=ududff,style={minimum size=5}] (w1) at (5,0) {};
	\node[label=north:\(\),fill=ududff,style={minimum size=5}] (w2) at (3,0) {};

%

	\Vertex[L=\hbox{$sK_1$},x=7,y=0]{w0}	
	\Vertex[L=\hbox{$(n-s-4)K_1$},x=10cm,y=0.0cm]{v0}
	
	\Edge[lw=0.1cm,style={post, right}](v0)(v1)
	\Edge[lw=0.1cm,style={post, right}](v1)(v2)
	\Edge[lw=0.1cm,style={post, right}](w0)(v0)
	\Edge[lw=0.1cm,style={post, right}](v1)(v0)
	\Edge[lw=0.1cm,style={post, right}](v2)(v1)
	\Edge[lw=0.1cm,style={post, right}](v0)(w0)
	\Edge[lw=0.1cm,style={post, right}](w1)(w2)
	\Edge[lw=0.1cm,style={post, right}](w0)(w1)
	\Edge[lw=0.1cm,style={post, right}](w1)(w0)
	\Edge[lw=0.1cm,style={post, right}](w2)(w1)

	\Edge[lw=0.1cm,style={post, bend left}](w2)(w0)
	\Edge[lw=0.1cm,style={post, bend right}](w2)(v0)
	\Edge[lw=0.1cm,style={post, bend right}](w1)(v0)
	
		\Edge[lw=0.1cm,style={post, bend right=90}](v2)(w0)
	\Edge[lw=0.1cm,style={post, bend left}](v2)(v0)
	\Edge[lw=0.1cm,style={post, bend right=90}](v1)(w0)

	\end{tikzpicture}

	\caption{The digraph $D_{n,r,s}$ for $r=3$}
	\label{fig:Dnrs}
\end{figure}

Just as in the graph case, like with Conjecture~\ref{conjchen}, for fixed $r$ there may be a few counterexamples for small $n$.
Also Conjecture~\ref{conjchendigraph} is asymptotically true.

\begin{thr}\label{maindi}
	For $r\ge 3$, there exists a value $n_1(r)$ such that for all $n \ge n_1(r)$ the following hold
	\begin{itemize}
		\item for any digraph $D$ or order $n$ with outradius $r$, we have $W(D) \ge W(D_{n,r,1})$. Equality holds if and only if $D \cong D_{n,r,s}$ for $1 \le s \le \frac{n-2r+2}{2}$.
	\end{itemize}
\end{thr}

The proof for this in Section~\ref{AsProofChendigraph} uses the same main ideas as in the graph case, but turns out to be slightly more difficult as the distance function is not symmetric. In this case we consider an expression which gives information about $W(D) - W(D \backslash v)$ in a more general setting.

We also prove some asymptotically sharp results on the maximum average distance given the radius. Here the behaviour depends on the exact notion of the radius used in the digraph case.
An overview of those results are presented in the Conclusion  in Tables~\ref{table1} and~\ref{table2}.

\subsection{Relation between maximum size and minimum average distance}

In this subsection, we explain the relation between maximizing the size and minimizing the average distance.

When one wants to minimize the average distance, one wants to have lots of small distances, so one may expect the minimum occurs when most distances are equal to one. That is, one can expect that the extremal graphs attaining the minimum average distance are the same as those attaining the maximum size.
When the diameter of the graph or digraph equals $2$, these two extremum problems are obviously equivalent.

Vizing~\cite{VZ67} determined the maximum size of a graph of given order and radius. Through the same arguments, one can check that the graphs $G_{n,r,s}$ are the only extremal graphs without cutvertices having the maximum size when $r\ge 3$.

\begin{thr}[\cite{VZ67}]\label{VZ}
	Let $f(n,r)$ be the maximum of edges in a graph with radius $r$.
	Then $f(n,1)=\binom{n}2, f(n,2)= \left \lfloor \frac{n(n-2)}2 \right \rfloor$ and 
$$f(n,r)=\frac{(n-2r)^2+5n-6r}{2} \mbox{ when } n \ge 2r \ge 6.$$
Equality occurs if and only if $G$ is a complete graph when $r=1$ or a complete graph minus a maximum matching (and an additional edge covering the remaining vertex when $n$ is odd) if $r=2.$
When $r \ge 3$ and $G$ has no cutvertex, equality occurs if and only if $G$ is isomorphic to $G_{n,r,s}$ for some $1 \le s \le \frac{n-2r+2}{2}$.
\end{thr}

So we see the extremal graphs are exactly the same as the ones conjectured by Chen et al~\cite{Chen}.
The graph $Q_3$ is an example showing that the family of extremal graphs are not exactly the same, as its size ($12$) is strictly smaller than the size of $D_{8,3,1}$ ($13$), but the total distance is the same ($48$).
It may be a possibility that the set of extremal graphs for one problem is a subset of the set of extremal graphs for the other problem.

The digraph version of Vizing's result for biconnected digraphs is trivial when $r \in \{1,2\}.$
When $r \ge 3$ and $n\ge 2r$, we conjecture that the extremal graphs are exactly those of the form $D_{n,r,s}$ for some $1 \le s \le \frac{n-2r+2}{2}$.

\begin{conj}\label{VZdi}
	Let $r \ge 3$ and $n \ge 2r$. Then the maximum size of biconnected digraphs with order $n$ and outradius $r$ is attained by $D_{n,r,1}.$ Furthermore the extremal digraphs are exactly the ones of the form $ D_{n,r,s}$ for some $1 \le s \le \frac{n-2r+2}{2}$.
\end{conj}

We note that in Theorem $5$ of \cite{F}, the maximum size of a digraph with out-radius $r$ has been determined in the case the digraph need not be biconnected.
In Section~\ref{sec:Vizing_digraph} we will prove Conjecture~\ref{VZdi} for $r=3$, as well as the case where $r \ge 4$ and $n$ is large enough.
Due to Theorem~\ref{main} and Theorem~\ref{maindi} we see that the extremal graphs agree for large $n$.

When one is working with the diameter instead of radius, the results are known and one can see similar relationships, when comparing the extremal (di)graphs attaining the minimum average distance (Plesn\'{i}k~\cite{P84}) and minimum size (Ore~\cite{Ore68}). The set of (di)graphs attaining the minimum average distance being a subset of the set of (di)graphs attaining the maximum size in this case.

The extremal (di)graphs in the latter case are formed by
taking two blow-ups at $2$ consecutive non-end vertices of a path of length $d$ (graph case) or the sum of a transitive tournament on $d+1$ vertices and the unique longest path in its complement (digraph case).
To get the extremal (di)graphs in the former case one needs to take blow-ups only at the 2 central vertices (or one central vertex).

\section{Notation and definitions}\label{not&def}

A graph will be denoted by $G=(V,E)$ and 
a digraph will be denoted by $D=(V,A).$
The order $\lvert V \rvert$ will be denoted by $n$. 
A clique or bidirected clique on $n$ vertices will be denoted by $K_n$. 
A cycle or directed cycle of length $k$ will be denoted by $C_k$.
The clique number of a graph $G$, $\omega(G)$, is the order of the largest clique which is a subgraph of $G.$
The complement $G^c$ of graph $G=(V,E)$ is the graph with vertex set $V$ and edge set $E^c=\binom{V}{2} \backslash E.$ 
The complement $D^c$ of a digraph $D$ is defined similarly, where the set of directed edges is het complement with respect to the edges of a bidirected clique.

The degree of a vertex in a graph $\deg(v)$ equals the number of neighbours of the vertex $v$, i.e. $\deg(v)=\lvert N(v) \rvert.$
In a digraph, we denote with $N^-(v)$ and $N^+(v)$ the (open) in- and outneighbourhood of a vertex $v$.
The indegree $\deg^-$ and outdegree $\deg^+$ of a vertex $v$, equals the number of arrows ending in or starting from the vertex $v$, i.e. $\deg^+(v)=\lvert N^+(v) \rvert $ and  $\deg^-(v)=\lvert N^-(v) \rvert $.
The total degree $\deg$ of a vertex $v$ in a digraph is the sum of the in- and outdegree, i.e. $\deg(v)=\deg^+(v)+\deg^-(v).$

Let $d(u, v)$ denote the distance between vertices $u$ and $v$ in a graph $G$ or digraph $D$, i.e. the number of edges or arrows in a shortest path from $u$ to $v$.  
The eccentricity of a vertex $v$ in a graph equals $\ecc(v)=d(v,V)=\max_{u \in V} d(v,u).$
The radius and diameter of a graph on vertex set $V$ are respectively equal to $\min_{v \in V} \ecc(v)$ and $\max_{v \in V} \ecc(v)=\max_{u,v \in V} d(u,v).$

In the case of digraphs, the distance function between vertices is not symmetric and so there is a difference between the inner- and outer eccentricity 
$\ecc^- (v)=d(V,v)=\max_{u \in V} d(u,v)$ and $\ecc^+(v)=d(v,V)=\max_{u \in V} d(v,u).$
We use the conventions as in e.g. \cite{JG}.
The in- and outradius of a digraph $D$ are defined by 
$\rad^{-}(D)= \min\{d(V,x)\mid x \in V\}$ and $\rad^{+}(D)= \min\{d(x,V)\mid x \in V\}$.
The radius of a digraph $D$ is defined as $\rad(D)=\min\{ \frac{d(x,V)+ d(V,x)}2 \mid x \in V \}.$ Sometimes authors refer to the outradius as radius, as outradius is the most common one between those three definitions.

The total distance, also called the Wiener index, of a graph $G$ equals the sum of distances between all unordered pairs of vertices, i.e. $W(G)=\sum_{\{u,v\} \subset V} d(u,v).$ 
The average distance of a graph is $\mu(G)=\frac{W(G)}{\binom{n}{2}}$. 
The Wiener index of a digraph equals the sum of distances between all ordered pairs of vertices, i.e. $W(D)=\sum_{(u,v) \in V^2} d(u,v).$
The average distance of the digraph is $\mu(D)=\frac{W(D)}{n^2-n}.$
A digraph is called biconnected if $d(u,v)$ is finite for any $2$ vertices $u$ and $v$. 

The statement $f(x)=O(g(x))$ as $x \to \infty$ implies that there exist fixed constants $x_0, M>0$, such that for all $x \ge x_0$ we have $\lvert f(x) \rvert \le M \lvert g(x) \rvert .$
Analogously, $f(x)=\Omega(g(x))$ as $x \to \infty$ implies that there exist fixed constants $x_0, M>0$, such that for all $x \ge x_0$ we have $\lvert f(x) \rvert \ge M \lvert g(x) \rvert.$
If $f(x)=\Omega(g(x))$ and $f(x)=O(g(x))$ as $x \to \infty$, then one uses $f(x)=\Theta(g(x))$ as $x \to \infty$. Sometimes we do not write the "as $x \to \infty$" if the context is clear.

\begin{defi}
	Given a graph $G$ and a vertex $v$, the blow-up of a vertex $v$ of a graph $G$ by a graph $H$ is constructed as follows.
	Take $G \backslash v$ and connect all initial neighbours of $v$ with all vertices of a copy of $H.$	
	When taking the blow-up of a vertex $v$ of a digraph $D$ by a digraph $H$, a directed edge between a vertex $w$ of $D \backslash v$ and a vertex $z$ of $H$ is drawn if and only if initially there was a directed edge between $w$ and $v$ in the same direction.
\end{defi}

Let $G_{n,r,s}$, where $n \ge 2r$ and $1 \le s \le \frac{n-2r+2}{2}$, be the graph obtained by taking two blow-ups of two consecutive vertices in a cycle $C_{2r}$ by cliques $K_s$ and $K_{n-2r+2-s}$ respectively.

Let $D_{2r,r,1}$ be a digraph with $2r$ vertices $v_1,v_2, \ldots v_r$ and $w_1, w_2, \ldots, w_r$, such that there are directed edges from $v_i$ to $v_j$ and from $w_i$ to $w_j$ if and only if $j \le i+1$ and a directed edge from any $v_i$ to $w_1$ and from any $w_i$ to $v_1.$

Let $D_{n,r,s}$, $n \ge 2r$ and $1 \le s \le \frac{n-2r+2}{2}$, be the graph obtained by taking the blow-up of $v_1$ by a bidirected clique $K_s$ and a blow-up of $w_1$ by a bidirected clique $K_{n-2r+2-s}$.

\section{Conjecture~\ref{conjchen} for large order}\label{AsProofChen}

We start with the calculation of the total distance of the graphs $G_{n,r,s}$ and $D_{n,r,s}$, which we want to prove to be extremal.

\begin{equation}\label{eq:1}
W(G_{n,r,s})=\binom{n}{2}+(r-1)^2n -r(r-1)^2
\end{equation}

\begin{equation}\label{eq:2}
W(D_{n,r,s})=2\binom{n}{2}+(r-1)^2 n-4 \binom{r}{3}.
\end{equation}

\begin{lem}\label{lem1}
	Suppose $G$ is a graph with order $n$ and total distance $W(G) < \binom{n}{2}+an$, for some constant $a$.
	Then $\omega(G) \ge \frac{n}{8a},$ i.e. $G$ contains a clique of order at least $\frac{n}{8a}$. Furthermore there exists such a clique such that all its vertices have degree at least equal to $n-4a.$
	
\end{lem}

\begin{proof}
	Let $S$ be the set of vertices of degree at least $n-4a.$
	Note that $\lvert S \rvert \ge \frac n2$ since otherwise $2W(G) \ge 2\binom{n}{2}+\frac{n}{2}4a.$
	Now the following algorithm returns a set $T$ of at least $\frac{n}{8a}$ vertices in $S$ which form a clique, since every time $T$ increases by one, $U$ decreases by at most $4a.$
\begin{lstlisting}
Start with $T=\emptyset$ and $U=S.$
While $U$ is nonempty, do:
	Take arbitrary $v \in U$ and set $T$:=$T \cup \{v\}$ and $U$:=$U \cap N(v)$
return $T$
\end{lstlisting}\end{proof}

\begin{lem}\label{lem2}
	Let $r \ge 3$.
	There is a value $n_0(r)$ such that for any $n \ge n_0$ and any graph $G$ of order $n$ and radius $r$ with minimal total distance among such graphs, there is a vertex $v \in G$ such that $G \backslash v$ has radius $r$ and the distance between any $2$ vertices of $G \backslash v$ equals the distance between them in $G.$
\end{lem}

\begin{proof}
	Note that such a graph $G$ satisfies $W(G)<\binom{n}{2}+a n$, where $a:=a(r)=(r-1)^2$ due to the example $G_{n,r,s}$ and Equation~\ref{eq:1}.
	Let $n_0:=n_0(r)= 192a^3$. By Lemma~\ref{lem1}, we know that $G$ contains a clique $K_k$ with $k \ge \frac{n}{8a}=24a^2$, such that the degree of all its vertices is at least $n-4a.$
	We can take a subset $S$ of vertices of $K_k$ of size at most $ 16a^2$ such that for any $2$ vertices $v,w$ in $G \backslash K_k$ which are connected with a common neighbour in $K_k$ have a common neighbour in $S$ and any vertex $v$ in $G \backslash K_k$ with a neighbour in $K_k$ has a neighbour in $S.$

	Taking a first vertex $x$ in $K_k$, it has degree at least $n-4a$ and hence it is connected to all vertices except for $s \le 4a-1$ vertices $v_1, v_2, \ldots, v_s$ in $G \backslash K_k.$
	For every of the $s$ vertices $v_i$ which has a neighbour $x_i$ in $K_k$, add a single $x_i$ to $S.$
	For every $v_i$, $1 \le i \le s$, there are at most $4a-1$ vertices which are not neighbours of $x_i$. For any such one which has a common neighbour in $K_k$ with $v_i$, add the common neighbour to $S$.
	We end with a set $S$ of size at most $4a(4a-1)+1<16a^2$ which satisfies the properties.

	Now for any vertex $v \in K_k \backslash S$, the distance measure in $G \backslash v$ equals the restriction to $G \backslash v$ of the distance measure in $G.$
	The radius of $G \backslash v$ is at least $r-1$, with equality if and only if there is at least one vertex $w_v$ in $G \backslash K_k$ such that its eccentricity as a vertex in $G \backslash v$ equals $r-1$ and $d_G(v,w_v)=r$.
	Let $T$ be the set of all vertices $v \in K_k \backslash S$ for which this happens. 
	Note that
	$$an>W(G)-\binom{n}{2} \ge \sum_{u \in K_k, v \in T} \left( d(u,w_v)-1 \right) \ge \lvert T \rvert k \mbox{, so }\lvert T \rvert < 8a^2.$$
	Since $k \ge 24a^2$, we can choose a vertex $v \in K_k \backslash (S \cup T).$ \end{proof}

\begin{lem}\label{lem3}
	Let $G$ be a graph with radius $r$ and order $n$ such that there is some vertex $v \in G$ such that $G \backslash v$ has also radius $r$ and the same restricted distance function.
	Then $W(G) \ge W(G \backslash v)+n-1+(r-1)^2    .$
	Equality occurs if and only if there is a path $\Q=w_{r}w_{r-1}\ldots w_1 u_1 u_2 \ldots u_r$ as subgraph of $G$, where $v=w_1$, such that $d_G(w_r,u_1)=d_G(w_1,u_r)=r$ and $d(v,w)=1$ for every vertex $w \in G$ which is not on this path.
\end{lem}

\begin{proof}
	Let the eccentricity of $v$ in $G$ be $r' \ge r.$
	This implies that there exists a path $\P = vu_1u_2 \ldots u_{r'}$ in $G$ which is the shortest path between $u_{r'}$ and $v=u_0$.
	If $r' \ge 2r-1$, then $\sum_{u \in G,u \not=v} d(u,v)-1 \ge \frac{(r'-1)r'}{2}>(r-1)^2$ and the statement is true.
	Since $G$ has radius $r$, there exists a vertex $z$ with $d(u_{r'-r+1},z)=r$.
	Since we now assume $r'<2r-1$, $z$ is not a vertex of $\P.$
	Take a shortest path $\P'$ from $v$ to $z$ and let $u_i$ be the vertex in $\P \cap \P'$ with the largest index. Note that $i<r'-r+1$ since $\ecc(v)=r'$.
	Then $d(u_{r'-r+1},u_i)+d(u_i,z)  	\ge d(u_{r'-r+1},z)=r$, i.e. $d(u_i,z) \ge r-(r'-r+1-i)=2r+i-r'-1$.
	We now get that
	\begin{align*}
	\sum_{u \in G,u \not=v} d(u,v)-1 &\ge \sum_{j=1}^{r'} (j-1) + \sum_{j=i+1}^{2r+i-r'-1} (j-1)\\
	&=\frac{(r'-1)r'}{2}+(2r-r'-1)i+\frac{(2r-r'-2)(2r-r'-1)}{2}\\
	&\ge \frac{r'^2+(2r-r'-1)^2 -2r+1}{2}\\
	&\ge \frac{(r'-1)^2+\left(2r-2-(r'-1)\right)^2 }{2}\\
	&\ge (r-1)^2.
	\end{align*}
	Here we used $2r-1>r'\ge r$, $i \ge 0$ and the inequality between the quadratic and arithmetic mean (QM-AM). Equality occurs if and only if $r'=r$, $i=0$, $d(v,z)=r-1$ and $d(w,v)=1$ for every vertex $w$ which is not part of $\P$ nor of $\P'$. 
\end{proof}

\begin{proof}[Proof of Theorem~\ref{main}]
	Take $a=a(r)=(r-1)^2$ and $n_1:= n_1(r)=n_0(r)+a(r)n_0(r).$ Let $n \ge n_1.$
	By Lemma~\ref{lem2} we know that there is a sequence of $n-n_0 \ge an_0$ vertices $v_{n_0+1}, \ldots, v_{n}$ which are consecutively added, starting from some graph $H$ with order $n_0$ and radius $r$, such that distances and the radius do not change.
	If no addition of such a vertex gives equality in Lemma~\ref{lem3}, then 
	
\begin{align*}
	W(G) &\ge W(H)+\sum_{i=n_0+1}^{n} \left( i+(r-1)^2 \right)\\
	&> \binom{n_0}{2} + \sum_{i'=n_0}^{n-1} i' +(n-n_0)+(n-n_0)(r-1)^2\\
	&\ge  \binom{n}{2}+an_0+ (n-n_0)a\\
	&= \binom{n}{2}+an
\end{align*}
	which is a contradiction as $G_{n,r,1}$ has a smaller total distance.
	
	Let $v=v_m$ be the first vertex whose addition gives equality in Lemma~\ref{lem3} and let $G_m$ be the final graph at that step.
	We know part of the characterization of this graph $G_m.$
	Furthermore, we note that $u_i$ is not connected with $w_j$ when $j<i$ as otherwise $d(w_1,u_i)=j<i$.
	Similarly, $u_i$ is not connected with $w_j$ when $j>i$ as otherwise $d(u_1,w_j)=i<j$.
	Also vertices $u_i$ and $w_i$ are not connected when $1<i \le r-1$ as otherwise $\ecc(w_i)<r$.
	For any neighbour $w$ of $v=w_1$, it is easy to see that we need $N[w] \cap \Q \subset \{w_3,w_2,w_1,u_1,u_2\}$.
	A small case distinction shows that $N[w] \cap \Q$ is part of one of the sets $\{w_3,w_2,w_1\},\{w_2,w_1,u_1\}, \{w_1,u_1,u_2\}.$
	Also if $2$ vertices $x,y$ of $\Q$ satisfy $N[x] \cap \Q=\{w_3,w_2,w_1\}$ and $N[y] \cap \Q= \{w_1,u_1,u_2\}$, they cannot be connected, since otherwise $\ecc(y)<r.$
	Now $\sum_{x,y \in \Q} d(x,y) \ge W(C_{2r})=r^3$,
	$\sum_{x,y \in G_m \backslash \Q} d(x,y) \ge W(K_{m-2r})$ and
	$\sum_{x\in \Q,y \in G_m \backslash \Q} d(x,y) \ge (m-2r)(r^2+1)$.
	
	Using the previous observations, we conclude that the graph with minimal total distance at this point (i.e. equality in the three above estimates) was some $G_{m,r,s}$ and the minimum attained after having added $v_{m+1},\ldots, v_n$ will give some $G_{n,r,s}.$ 
\end{proof}

\section{Conjecture~\ref{conjchendigraph} for large  order}\label{AsProofChendigraph}

We first prove two lemmas which are formulated more generally.

\begin{lem}\label{lem:g1}
	Let $D$ be a digraph with average total degree at least being equal to $2(n-1)-2t$. Then at least half of the vertices have total degree more than $2(n-1)-4t.$ 
	Also $\omega(D) \ge \frac{n}{8t},$ i.e. $D$ contains a bidirected clique of size at least $\frac{n}{8t}$.
\end{lem}

\begin{proof}
	The first part is trivial, as the contrary would lead to a contradiction.
	The second part is analogous to the proof of Lemma~\ref{lem1}, with $S$ now being the set of vertices with total degree more than $2(n-1)-4t.$
	In the algorithm, $N(v)=N^+(v) \cap N^-(v),$ will denote the set of vertices $w$ which are both in-neighbours and out-neighbours of $v$.
\end{proof}

\begin{lem}\label{lem:g2}
	Let $D$ be a digraph with outradius $r\ge 3$, order $n$ and size at least $n(n-1-t)$ such that it contains a bidirected clique $K_k$ for which all of its vertices have total degree at least $2(n-1)-4t.$ 
	If $k>32t^2+4t+\frac{tn}k$, there is a vertex $v \in K_k$ such that $D \backslash v$ has outradius $r$ and the distance between any $2$ vertices of $D \backslash v$ equals the distance between them in $D.$
\end{lem}

\begin{proof}
	We will first construct a set $S$ of vertices of $K_k$ of size at most $32t^2+4t$
	such that for any $2$ vertices $x,y$ in $D \backslash K_k$ for which there exists a vertex $v \in K_k$ such that $\vec{xv}$ and $\vec{vy}$ are edges of $D$, there is an $s \in S$ with $\vec{xs}$ and $\vec{sy}$ being edges of $D$ as well.
	Take a first vertex $s_1$ of $K_k$ and assume $Z$ is the set of vertices $z$ of $D \backslash K_k$ such that there are no edges in both directions between $s_1$ and $z.$ 
	For any vertex $z \in Z$ that has an edge towards $K_k$, take an additional vertex $s_i \in K_k$ (which we put in $S$) such that there is an edge from $z$ to $s_i$. Similarly for every vertex $z \in Z$ such that there is an edge from $K_k$ to $z$, we take some $s_i \in K_k$ for which there is an edge from $s_i$ to $z$.
	Note at this point $\lvert S \rvert \le 4t.$ 
	Now there are less than $2\cdot 4t \cdot 4t=32t^2$ pairs of vertices $(x,y)$ in $D \backslash K_k$ such that the property is not satisfied by $S$ yet.
	Adding a corresponding vertex of $K_k$ to $S$ gives a set $S$ satisfying the property.
	Now for any vertex $v \in K_k \backslash S$, the distance measure in $D \backslash v$ equals the restriction to $D \backslash v$ of the distance measure in $D.$
	The outradius of $D \backslash v$ is at least $r-1$. In case equality holds, there is at least one vertex $w_v$ in $D \backslash K_k$ such that $d(w_v,v)=r$ and the outer eccentricity of $w_v$ as a vertex in $D \backslash v$  equals $r-1$.
	Let $T$ be the set of all vertices $v \in K_k \backslash S$ for which this happens. 
	Then for every $v \in T$, there is an associated $w_v$ which is not associated with another element from $T$, for which there is no edge from $w_v$ to any element of $K_k$.
	This implies that at least $\lvert T \rvert k$ arrows are missing, which has to be at most $tn$, i.e. $\lvert T \rvert \le \frac{tn}k.$\\
	Since $k\ge 32t^2+4t+\frac{tn}k$, we can choose a vertex $v \in K_k \backslash (S \cup T).$
\end{proof}

\begin{lem}\label{lem2di}
	Let $r \ge 3$.
	There is a value $n_0(r)$ such that for any $n \ge n_0$ and any digraph $D$ of order $n$ and outradius $r$ with minimal total distance among such digraphs, there is a vertex $v \in D$ such that $D \backslash v$ has outradius $r$ and the distance between any $2$ vertices of $D \backslash v$ equals the distance between them in $D.$
\end{lem}

\begin{proof}
	Note that such a graph $D$ satisfies $W(D)<2\binom{n}{2}+a n$, where $a:=a(r)=(r-1)^2$ due to the example $D_{n,r,s}$ and Equation~\ref{eq:2}.
	If the size of the digraph would be smaller than $n(n-1-a)$, then $W(D) \ge n(n-1-a)+ 2an=2\binom{n}{2}+a n$, which is a contradiction. 
	Due to the equivalent of the handshaking lemma, we know the average total degree is at least $2(n-1)-2a$ and hence by Lemma~\ref{lem:g1}  we know that $D$ contains a clique $K_k$ with $k \ge \frac{n}{8a}$, such that the (total) degree of all its vertices is at least $2(n-1)-4a.$
	Let $n_0:=n_0(r)= 8a(40a^2+4a)$.
	Since $k \ge \frac{n}{8a} > 32a^2+4a+\frac{an}{k}\ge 32a^2+4a+8a^2$ when $n>n_0$, the result follows from Lemma~\ref{lem:g2}.
\end{proof}

\begin{prop}\label{proplem}
	Let $D=(V,A)$ be a digraph, $v \in V$ a vertex with outeccentricity $\ecc^+(v)=r' \ge r\ge 3$.
	Let $\P=vu_1u_2\ldots u_{r'}$ be a directed path in $D$ with $d(v,u_{r'})=r'$ such that for every vertex $u_i$, $r'-r+1 \le i \le r'$, there is a vertex $x_i$ with $d(v,x_i)<d(v,u_i)+d(u_i,x_i)$ such that $d(u_i,v)+d(v,x_i)\ge r$.
	Then \begin{equation}\label{Part1}
		\sum _{u \in V, u \not =v}\left(d(v,u)-1\right) + \sum _{1 \le i \le r'} \left(d(u_i,v)-1\right) \ge (r-1)^2 \end{equation} with equality if and only if the following conditions (up to labeling) are satisfied
	\begin{itemize}
		\item $r'=r$,
		\item there is a second directed path $v w_2w_3\ldots w_{r}$ with $d(v,w_{r})=r-1$ which is disjoint from the first directed path up to the vertex $v$,
		\item $d(u_i,v)=1$ for every $1 \le i \le r$, and
		\item  $d(v,u)=1$ for all vertices $u$ not on those two directed paths.	
	\end{itemize}
\end{prop}

\begin{proof}
	For any $r'-r+1 \le i \le r'$, we take $x_i$ being equal to a vertex on the path $\P$ if possible (that is when $d(u_i,v)+i-1 \ge r$).
	If this is not possible we take a shortest (directed) path $\Q$ from $v$ to $x_i$ which has the least number of vertices in common with $\P$. Let $\P \cap \Q$ be the directed path $vu_1 \ldots u_j$.
	Then we say $x_i$ is of type $j$.
	First we notice that there will be no $x_i$ of type $j$ for any $j \ge 2$.
	For this, choose the $x_i$ of type $j$($\ge 2$) with $d(u_j,x_i)$ maximal. 
	Note that we can choose $x_k=x_i$ only if $k$ is between $j+1$ and $j+d(u_j,x_i)$ and for no other $k$.
	So if we can adapt $D$ such that $x_i$ is deleted and $d(u_k,v)$ increases by one for every $j+1 \le k \le j+d(u_j,x_i)$, the left hand side (LHS) of expression~\ref{Part1} decreases by $j-1$.
	If we have some $x_i$ of type $1$, we again choose $x_i$ with $d(u_1,x_i)$ maximal. We adapt the path $\Q=vu_1w_3w_4\ldots x_i$ from $v$ to $x_i$ to 
	$v w_2w_3 \ldots x_i$.
	Note that $r'>r$, or $r'=r$ and we can decrease $d(u_1,v)$.\\
	So having no $x_i$ of type $j \ge 1$, we know that there is some directed path $\Q=vw_2w_3 \ldots w_s$ for some $s \le r'$ which is disjoint from $\P$ (except from $v$) or we could take $x_i=u_{i-1}$ in all cases.\\
	In the second case, we have that the LHS of Equation~\ref{Part1} is at least 
	$$\frac{r'(r'-1)}{2}+\frac{(2r-r')(2r-r'-1)}{2}=\frac{(r'^2+(2r-r')^2)}{2}-r \ge r^2-r>(r-1)^2,$$
	where we used the inequality between the quadratic and arithmetic mean (QM-AM).\\
	In the first case, we have that $d(u_i,v)+s-1 \ge r$ when $r'-r+1 \le i \le s-1$ where we can assume $2 \le s \le r$, as $s\ge r+1$ implies the result trivially.
	When $s \le i \le r$, we have $d(u_i,v)+d(v,u_{i-1}) \ge r$.
	So we get that the LHS of Equation~\ref{Part1} is at least 
	\begin{align*}
	&\sum_{i=1}^{r'} \left(d(v,u_i)-1 \right) + \sum_{i=2}^{s} \left(d(v,w_i)-1 \right) + \sum_{i=r'-r+1}^{s-1} \left(d(u_i,v)-1 \right) + \sum_{i=s}^{r} \left(d(u_i,v)-1 \right)\\&\ge 
	\frac{r'(r'-1)}{2}+\frac{(s-1)(s-2)}{2}+(s+r-r'-1)(r-s)+\frac{(r-s)(r-s+1)}{2}.	\end{align*}
	This expression is strictly increasing for $r' \ge r$, so the minimum is attained when $r'=r$. So the expression reduces to $r^2-r-s+1$ which is minimal for $s=r$ and gives exactly $(r-1)^2.$ 
	In case of equality, we also need $d(v,u)=1$ for vertices $u$ not considered (not on the paths), from which the characterization of the equality constraints is clear as well.	
\end{proof}

\begin{lem}\label{D2rr1core}
	Let $D=(V,A)$ be a digraph with outradius $r \ge 3$ containing directed paths 
	$w_1u_1u_2\ldots u_{r}$ and $w_1 w_2w_3\ldots w_{r}$ which are vertex-disjoint up to $w_1$ with $d(w_1,u_r)=r$ and $d(w_1,w_r)=r-1$.
	Let $V_1=\{w_r,w_{r-1} \ldots, w_2,w_1,u_1,u_2,\ldots, u_{r}\}$.
	Assume $d(u,w_1)=1$ for all $u \in V$ and $d(w_1,u)=1$ for all $u \in V \backslash V_1$.
	Then 
	 \begin{equation}\label{Part2}
		\sum _{x,y \in V_1} d(x,y) \ge W(D_{2r,r,1}) \end{equation} with equality if and only if $D[V_1]$ is isomorphic to $D_{2r,r,1}$ and for any $y \in V \backslash V_1$ we have
		\begin{equation}\label{Part3}
		\sum_{x \in V_1} \left(d(x,y)+d(y,x)-2\right) \ge (r-1)^2. \end{equation}

\end{lem}

\begin{proof}[Proof of (\ref{Part2})]
	First note that for every $u \in V \backslash V_1$ and $v \in V$, we have $d(v,u) \le d(v,w_1)+d(w_1,u)\le 2$ and so $\ecc^+(v)\ge r$ implies there is a $x \in V_1$ with $d(v,x)\ge r$.
	
	By the given conditions, we know that $d(u_i,u_j)=d(w_i,w_j)=j-i$ when $1 \le i \le j \le r$ and $d(w_i, w_j),d(u_i,u_j) \ge 1$ if $j < i$.\\
	We also have $d(u_i,w_j)=j$ for all $1 \le i \le r-1$, $j$ being an upper bound since $d(u_i,w_1)=1$ and $d(w_1,w_j)=j-1$, while $d(u_i,w_j)<j$ for some $j$ would imply $d(u_i,w_r)<r$ and hence $\ecc^+(u_i)<r$, which would be a contradiction.
	
	Next, we see $d(w_i,u_j) \ge j$ for all $1 \le j \le r-1$.
	When $\min\{i,j\}=1$, this is by definition.
	When  $d(w_i,u_j) < j$ for some $1<j<r$ and $1<i$, we would get $\ecc^+(w_i)<r$, since $d(w_i,w_k)<r$ for all $k$, $d(w_i,u_k)\le 1+k<r$ when $k<j$ and $d(w_i,u_k)\le (j-1)+(k-j)=k-1 \le r-1$ when $k \ge j.$
	
	So the digraphs for which Equation~\ref{Part2} is not immediate, satisfy $d(u_r,w_j)<j$ for some $j$ or $d(w_i,u_r)<r$ for some $i$.
	Since $d(w_1,u_r)=r$, $d(w_1,w_i)=i-1$ and $d(w_i,u_{r-1})\ge r-1$, we see that $d(w_i,u_r)<r$ is only possible if there is an arc from $w_r$ to $u_r$.
	The condition $d(u_r,w_j)<j$ for some $j>1$ would imply $d(u_r,w_j)=j-1$ since otherwise $\ecc^+(u_{r-1})<r$ and hence $d(u_r,u_{r-1})=r$ as $\ecc^+(u_r)=r$ and so $d(u_r,u_{i})\ge 1+i$ for all $1 \le i \le r-1.$
	Note this already implies that if $d(w_i,u_r)=r$ for all $i$, Equation~\ref{Part2} is satisfied.
	
	So assume $d(w_r,u_r)=1$ and remark we now need $d(w_i, u_{r-1})=r$ for every $1<i \le r$,
	so $d(w_i,u_j) \ge j+1$ for all $1 \le j \le r-1$ and  $1<i \le r$.
	So at least $(r-1)^2$ distances are at least $1$ larger than the corresponding distances in $D_{2r,r,1}$, while the edge between $w_r$ and $u_r$ made we won only $\frac{r(r-1)}{2}$ with the terms corresponding to $d(w_i,u_r)$ for $2 \le i \le r$.
	Taking into account all possibilities, we see Equation~\ref{Part2} is always true and equality is possible if and only if $D[V_1]$ is isomorphic to $D_{2r,r,1}$.
\end{proof}	
	
\begin{proof}[Proof of (\ref{Part3})]
	Note that $d(y,w_r) \ge r-2$ and $d(y,u_r)\ge r-1$ due to the conditions on $w_1$.
	As $\ecc^+(y)\ge r$, there is at least one vertex at distance $\ge r$, the only possible vertices for this are in $\{w_r,u_{r-1},u_r\}$.
	If $d(y,w_r)=r$, combining this with $d(y,u_r) \ge r-1$ we already have $\sum_{x \in V_1} \left(d(y,x)-1\right) \ge (r-1)^2.$
	The same holds if $d(y,u_r)=r$ and $d(y,w_r)\ge r-1$ or if $d(y,u_{r-1})=r$ as then the shortest path from $y$ to $u_r$ goes by some $w_i$ or $d(y,u_r)=r+1$ and $d(y,w_r) \ge r-2.$
	In the last case $d(y,u_r)=r$ and $d(y,w_r)=r-2$. But this implies $d(u_i,y)\ge 2$ for every $1\le i \le r-1$ as otherwise $\ecc^+(u_i)<r$.
	So now we have
	$$\sum_{x \in V_1} \left(d(x,y)-1\right) \ge r-1 \mbox{ and }
	\sum_{x \in V_1} \left(d(y,x)-1\right) \ge \frac{r(r-1)}{2}+\frac{(r-2)(r-3)}{2}$$
	 from which the result follows again as the sum is at least $(r-1)^2+1.$
\end{proof}

\begin{proof}[Proof of Theorem~\ref{maindi}]
	
Let $a=a(r)=(r-1)^2$ and $n_1:= n_1(r)=n_0(r)+a(r)n_0(r).$ Let $n \ge n_1.$
By Lemma~\ref{lem2di} we know that there is a sequence of $n-n_0 \ge an_0$ vertices $v_{n_0+1}, \ldots, v_{n}$ which are consecutively added, starting from some digraph $H$ with order $n_0$ and radius $r$, such that distances and the radius do not change.
Note that a digraph of radius $r$ satisfies the conditions of Proposition~\ref{proplem}.
If no addition of any of the $an_0$ vertices gives equality in Proposition~\ref{proplem}, then
\begin{align*}
W(D) &\ge W(H)+\sum_{i=n_0+1}^{n} \left(2i-1+(r-1)^2\right) \\
&> 2\binom{n_0}{2} + 2\sum_{i'=n_0}^{n-1} i' +(n-n_0)+(n-n_0)(r-1)^2\\
&\ge  2\binom{n}{2}+an_0+ (n-n_0)a\\
&= 2\binom{n}{2}+an
\end{align*}
which is a contradiction as $D_{n,r,1}$ had a smaller total distance (Equation~\ref{eq:2}).
So we have equality in some step adding $v_m$ in Proposition~\ref{proplem} and we get a digraph $D_m$ at that step. Knowing the conditions of equality of Proposition~\ref{proplem} and that $D_m$ has outradius $r$, we may apply Lemma~\ref{D2rr1core} to conclude that $D_m$ should be of the form $D_{m,r,s}$ for some $s$ and so does the digraph at the final step.
\end{proof}

\subsection{Minimum for digraphs given order and radius}

For small $r$, we easily can determine the exact minimum Wiener index of digraphs with given order and radius $r$. Note that $r$ can be an integer or a half-integer, i.e. $1 \le r$ with $r \in \frac12 \mathbb Z.$

\begin{prop}
	The minimum Wiener index among all digraphs $D$ with radius $r$ and order $n$ is at least
	$$
	\begin{cases}
	2\binom{n}{2} \hfill \mbox{ if } r=1,\\
	2\binom{n}{2}+\lceil \frac n2 \rceil \hfill \mbox{ if }r=\frac32 \mbox{ or}\\
	n^2 \hfill \mbox{ if }r=2 .\\
	\end{cases}
	$$
	Equality holds if and only if $D=K_n$, $D^c$ is the union of $\lceil \frac n2 \rceil$ directed edges which are spanning or $D^c$ is the union of some vertex disjoint directed cycles which span all vertices.
\end{prop}

When $r \ge \frac 52$, the analog of Lemma~\ref{lem2di} hold. Nevertheless, the analog of Proposition~\ref{proplem} does not give a unique configuration and so we only conclude with the following asymptotic result.

\begin{thr}\label{min_rad}
	For $r \ge \frac 52$, the minimum Wiener index among all digraphs with radius $r$ and order $n$ is of the form $2\binom{n}2 + \lfloor (r-0.5)^2 \rfloor n + \Theta_r(1).$
\end{thr}

\begin{proof}
	By definition of the radius, for every vertex $x \in V$, we have $d(x,V)+d(V,x)\ge 2r$. Let $a=d(x,V)$ and $b=d(V,x).$
	Then $\sum_{v \in V \backslash x} \left( d(x,v)-1+d(v,x)-1 \right) \ge \frac{a(a-1)}2+\frac{b(b-1)}2 \ge \lfloor (r-0.5)^2 \rfloor.$

	When $r \in \mathbb N$, we can take the blow-up of a vertex of a directed cycle $C_{r+1}.$
	When $r \in \frac{1}{2}+\mathbb N$, we can take a directed cycle $C_{r+1.5}$ with an additional directed edge in the opposite direction between two neighbours.
	Now take a blow-up of the startvertex of that additional directed edge.
\end{proof}

Note that taking a blow-up by a clique $K_{n-4}$ in vertex $4$ or $2$ of the digraph in Figure~\ref{fig:digraph_rad} gives a digraph with a smaller total distance than the blow-up of a vertex of a $C_{r+1}$ when $r=3$.

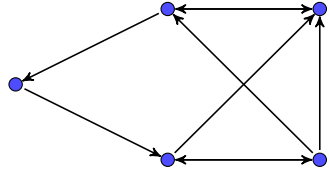
\begin{figure}[h]
	\centering

\begin{tikzpicture}

\node[label=center:\(\),fill=ududff,style={minimum size=5,shape=circle}] (v0) at (4,1) {};
\node[label=center:\(\),fill=ududff,style={minimum size=5}] (v1) at (2,1) {};
\node[label=center:\(\),fill=ududff,style={minimum size=5,shape=circle}] (v2) at (0,0) {};
\node[label=center:\(\),fill=ududff,style={minimum size=5,shape=circle}] (v3) at (2,-1) {};
\node[label=center:\(\),fill=ududff,style={minimum size=5,shape=circle}] (v4) at (4,-1) {};

\Edge[lw=0.1cm,style={post, right}](v0)(v1)
\Edge[lw=0.1cm,style={post, right}](v1)(v0)
\Edge[lw=0.1cm,style={post, right}](v1)(v2)
\Edge[lw=0.1cm,style={post, right}](v4)(v1)
\Edge[lw=0.1cm,style={post, right}](v4)(v3)
\Edge[lw=0.1cm,style={post, right}](v2)(v3)
\Edge[lw=0.1cm,style={post, right}](v3)(v4)
\Edge[lw=0.1cm,style={post, right}](v3)(v0)

\Edge[lw=0.1cm,style={post, right}](v4)(v0)
\end{tikzpicture}
	\caption{Digraph with $\rad=3$ and small Wiener index}
\label{fig:digraph_rad}
\end{figure}

\section{Maximum size biconnected digraphs}\label{sec:Vizing_digraph}

In this section, we prove that Conjecture~\ref{VZdi} is true when $r=3$, as well as in the case $r>3$ and $n$ large enough with respect to $r$.

\begin{thr}
	Let $n \ge 6.$ Then any biconnected digraph $D=(V,A)$ with order $n$ and outradius $3$ satisfies $\lvert A \rvert \le (n-2)^2$. Equality holds if and only if $ D \cong D_{n,r,s}$ for some $1 \le s \le \frac{n-2r+2}{2}$.
\end{thr}

\begin{proof}
	Note that $\lvert A \rvert = \sum_{v \in V} d^+(v).$
	Since the outradius of $D$ is $3$, we know $\ecc^+(v) \ge 3$ for every $v \in V$, which implies that $d^+(v) \le n-3$ for every $v \in V$.
	If $\lvert A (D) \rvert \ge n^2-4n+	4$, we know there are at least $4$ vertices with outdegree equal to $n-3.$
	For any vertex $v$, let $F(v)$ be the set of vertices $x$ in $V$ different from $v$ for which there is no arrow from $v$ to $x$, i.e. $F(v)=V \backslash N^+[v].$ Note that $d^+(v)=n-1-\lvert F(v)\rvert.$
	Let $c$ (being a center of the digraph) have outdegree $n-3$ and let $a_2$ and $a_3$ be the $2$ vertices such that $d(c,a_2)=2$ and $d(c,a_3)=3.$
	Let $Y$ be the set of vertices $y$ such that $d(c,y)=1=d(y,a_2)$ and $X$ be the remaining vertices. Note that $X$ is not empty, as otherwise $D$ has at most the size of a digraph formed by taking a blowup of a vertex of a directed $C_4$, which has a smaller size than $(n-2)^2.$ For this, remark that there cannot be a directed edge from some $y \in Y$ or $a_2$ to $c$ as then the outereccentricity of that vertex is at most $2$. So there is a directed edge from $a_3$ to $c$.
	For the same reason, there cannot be directed edges from $a_3$ to $a_2$ or some $y \in Y$.
	It is possible there is an edge from $a_2$ to some $y \in Y$, but then these vertices cannot have $Y \subset N^+[y]$ and hence the outdegree of these vertices is lower than in the blowup of the directed $C_4,$ from which the bound on the size follows.
	
	Now we do some case analysis.
		
	\textbf{Case 1} There is a $y \in Y$ with $d^+(y)=n-3$.
	Note that there is no directed edge from $y$ to $c$ since otherwise $\ecc^+(y)=2$ and so there is an arrow from $y$ to all vertices different from $c$ and $a_3,$ i.e. $F(y)= \{a_3,c\}.$
	Now we see $\{a_2,a_3,c\} \subset F(x)$ for all $x \in X$ (using $d(y,c)=3$ and $d(y,x)=1$), $\{c,y\} \subset F(a_2)$, $Y \cup \{a_2\} \subset F(a_3)$ and $\{c,a_3\} \subset F(y')$ for all $y' \in Y$ different from $y$. From this we see $\lvert A (D) \rvert \le n^2-4n+	4$ and equality is not possible, since then $\ecc^+(a_2)=2.$
	
	\textbf{case 2} The set $X_2$ containing the vertices $x \in X$ with $d^+(x)=n-3$ has size at least $2$, note that we have $F(x)=\{a_2,a_3\}$ for those $x$.
	Furthermore we remark that $X_2 \cup \{c, a_3\} \subset F(y)$ for every $y \in Y$ and $ X_2 \cup \{c\} \subset F(a_2)$.
	We get $\lvert A (D) \rvert \le n^2-4n+	4- \lvert Y \rvert \left( \lvert X_2 \rvert -1 \right) $ in this case.
	
	\textbf{case 3} In the last case, there is one vertex $x_2 \in X$ with $d^+(x_2)=n-3$.
	Furthermore we need $d^+(a_2)=d^+(a_3)=d^+(c)=n-3$ and $d^+(v)=n-4$ for all the remaining vertices $v$.
	We see that $F(x_2)=F(c)=\{a_2, a_3\}$ and $F(a_2)=\{c, x_2\}.$
	
	Also $F(y)=\{a_3,c,x_2\}$ as $\ecc^+(y) \ge 3$ for every $y \in Y.$
	If $X=\{x_2\},$ then we get a contradiction as there should be a directed edge from $a_3$ to exactly one of $c$ and $x_2$ and consequently there cannot be a directed edge to some $y \in Y$ or to $a_2$.
	So now assume $X\not=\{x_2\}.$ By definition we already have $\{a_2,a_3\} \subset F(x)$ for any $x \in X.$
	Since $\ecc^+(a_2)=3$ and $F(a_2)=\{c, x_2\},$ we see that there can't be directed edges from $X \backslash \{x_2\}$ to both $c$ and $x_2$.
	Wlog there is no directed edge towards $c$, i.e. $F(x)=\{a_2,a_3,c\}$ for every $x \in X \backslash \{x_2\}$.
	Finally, we see $F(y)=\{c,x_2\}$ since there cannot be a directed edge towards both $c$ and $x_2$ and if there is exactly one, then we need $d(a_3,a_2)=3$ which cannot either.
	So we conclude $D$ is isomorphic to $D_{n,r,s}$ where $s=\min \{ |X|, |Y|\}.$
	\end{proof}

Now, we prove that Conjecture~\ref{VZdi} holds when $n$ is large enough wrt to a fixed $r$.
First, we note that the analog of Lemma~\ref{lem2di} holds.

\begin{lem}\label{lem2di_size}
	Let $r \ge 3$.
	There is a value $n_0(r)$ such that for any $n \ge n_0$ and any digraph $D=(V,A)$ of order $n$ and outradius $r$ with maximum size among such digraphs, there is a vertex $v \in D$ such that $D \backslash v$ has outradius $r$ and the distance between any $2$ vertices of $D \backslash v$ equals the distance between them in $D.$
\end{lem}

\begin{proof}
	Note that such a digraph $D$ satisfies $\lvert A \rvert > n(n-1-t)$, where $t:=t(r)=2r-3$ due to the example $D_{n,r,s}.$
	We know from Lemma~\ref{lem:g1} that $D$ contains a clique $K_k$ with $k \ge \frac{n}{8t}$, such that the (total) degree of all its vertices is at least $2(n-1)-4t.$
	Let $n_0:=n_0(r)= 8t(40t^2+4t)$.
	Since $k \ge \frac{n}{8t} >40t^2+4t \ge 32t^2+4t+\frac{tn}{k}$ when $n>n_0$, the result follows from Lemma~\ref{lem:g2}.
\end{proof}

\begin{prop}\label{proplem_size}
	Let $D=(V,A)$ be a digraph of order $n$ and outradius $r$.
	Then for any vertex $v$, the total degree $\deg(v)\le 2(n-1)-(2r-3).$
	Equality can occur if and only if
	\begin{itemize}
		\item  $d^-(v)=n-1$ and $d^+(v)=n-1-(2r-3)$,
		\item  there exists two disjoint directed paths $vu_1u_2\ldots u_{r}$ and $v w_2 w_3 \ldots w_r$ in $D$ with $d(v,u_r)=r$ and $d(v,w_r)=r-1$	
	\end{itemize}
\end{prop}

\begin{proof}
	Let $\ecc^+(v)=r' \ge r$ and assume $vu_1u_2\ldots u_{r'}$ is a directed path in $D$ with $d(v,u_{r'}) =r'.$
	Let $i$ be the smallest index with $r'-r+1 \le i \le r'$ for which there is an arrow from $u_i$ to $v.$ Note that this one has to exist, or otherwise $\deg(v) \le 2(n-1)-(2r-1).$
	Since $\ecc^+(u_{r'-r+1})\ge r$, there exists a vertex $x$ with $d(u_{r'-r+1},v)+d(v,x)\ge d(u_{r'-r+1},x) \ge r$. 
	Note that this implies $d(v,x) \ge r-(i-r'+r)=r'-i.$
	The shortest path from $v$ to $x$ does not pass $u_{r'-r+1}$ and so we have at least $r'-i-1$ vertices $w$ on this path for which there is no edge from $v$ to $w$.
	There is no edge from $v$ to $u_j$ when $r'-r+2 \le j \le r'$ either.
	Furthermore $v$ is not in the outneighbourhood of $v_j$ when $r'-r+1 \le j < i$.
	If $r'>r$, then $d(v,u_{r'-r+1})>1$ as well and in this case $\deg(v) \le 2(n-1)-(2r-2).$
	So now $r'=r$. If $i=r$, one has $\deg(v) \le 2(n-1)-(2r-2)$ again.
	If $1<i<r$, one gets a contradiction since $\ecc^+(u_i)<r$.
	In the remaining case of equality, we have the desired form.	
\end{proof}

\begin{thr}\label{mainVz_bicDi}
	For $r\ge 4$, there exists a value $n_1(r)$ such that for all $n > n_1(r)$ the following hold
	\begin{itemize}
		\item for any digraph $D$ or order $n$ with outradius $r$, we have $\lvert A(D) \rvert\le \lvert A(D_{n,r,1})\rvert =(n-(r-1))^2+(r-3)$, with equality if and only if $D \cong D_{n,r,s}$ for $1 \le s \le \frac{n-2r+2}{2}$.
	\end{itemize}
\end{thr}

\begin{proof}
	Let $n_1:=n_1(r)=(2r-2)n_0.$
	Take a digraph $D$ of size $n > n_1$ which has the maximum size among all digraphs of order $n$ and outradius $r.$
	By Lemma~\ref{lem2di_size} there exists a digraph $D'$ such that $D$ is formed by adding vertices $v_{n_0+1} \ldots v_n.$
	If for all of these additions, there is no equality in Proposition~\ref{proplem_size}, then as $n>(2r-2)n_0$ we have
	\begin{align*}
	\lvert A(D) \rvert \le n_0(n_0-1)+\sum_{i=n_0+1}^n 2(i-1)-(2r-2)  &=n^2-(2r-1)n+(2r-2)n_0\\ &< \lvert A(D_{n,r,1})\rvert.
	\end{align*} 
	So now assume there was equality in some step in Lemma~\ref{lem2di_size}. Then we know a partial characterization of the digraph $D_m=(V_m,A_m)$ in that step. 
	Let $B=\{u_1,u_2, \ldots, u_{r-1}, u_r, v, w_2, \ldots ,w_r\}$ and $R=V_m \backslash B$.
	Note that there is no directed edge from $u_i$ to $w_j$ if $1 \le i \le r-1$ and $2 \le j \le r$ as otherwise $\ecc^+(u_i)<r$.
	Similarly, there is no directed edge from $w_i$ to $u_j$ if $2 \le i \le r$ and $2 \le j \le r-1$ and at most one to $u_1$ or $u_r$.
	Since $\ecc^+(u_r) \ge r$, we need $d(u_r,u_{r-1})=r$ or $d(u_r,w_r)=r$.
	This implies that it is impossible there are directed edges from $u_r$ to both some $w_i$ and $u_j$ with $2 \le i \le r$ and $1 \le j \le r-1.$
	If $N^+(u_r) \cap B = \{v,w_2,w_3, \ldots, w_r\}$, then $\ecc^+(u_{r-1})<r$ which is a contradiction again.
	Knowing all of this, we already have an upper bound on the number of directed edges between vertices in $B$. 
	Next, let $z \in R$ be a remaining vertex.
	We now prove that there cannot be more than $2r+3$ directed edges between $z$ and $B.$
	Note that $N^+(z) \cap B \subset \{w_3,w_2,v,u_1,u_2\}$.
	If $w_3 \in N^+(z)$, then there is no $u_i$ with $1 \le i \le r-1$ in $N^-(z)$ and hence there are at most $r+6<2r+3$ directed edges between $z$ and $B$ in this case.
	Since $w_2$ and $u_2$ cannot be both in $N^+(z)$ as otherwise $\ecc^+(z)<r$, there are indeed at most $2r+3$ directed edges between $z$ and $B$.
	So indeed $ \lvert A_m \rvert  \le \lvert A(D_{m,r,1}) \rvert$ with equality if and only if $D_m$ is isomorphic to $D_{m,r,s}$ for a certain $s$.
	For this, note that if all equalities need to hold, we cannot have a directed edge from some $w_i$ to $u_r$, since then $d(w_i,u_{r-1})=2$ and hence $\ecc^+(w_i)<r$.
	For every vertex which is added, we know the total degree of that vertex is bounded again by Proposition~\ref{proplem_size} and equality occurs if and only if the digraph with this additional vertex is also of the form $D_{m',r,s'}$ for certain $m'$ and $s'$, from which the conclusion follows.
\end{proof}

\section{Maximum total distance}

The maximum total distance of graphs given their order and diameter was derived asymptotically in Theorem $3.1$ of \cite{SC19}.

\begin{thr}[\cite{SC19}]
	
	The maximum Wiener index of a graph with order $n$ and diameter $d \ge 3$ is $\frac{d}{2}n^2-d^{3/2}\Theta( n^{3/2}).$
\end{thr}
 As a corollary, we have the following theorem for the same question when considering the radius instead of the diameter.

\begin{cor}\label{cor_radius}
	The maximum Wiener index among graphs of order $n$ and radius $r$ equals $rn^2-r^{3/2}\Theta(n^{3/2}).$ The extremal graphs are trees.
\end{cor}

\begin{proof}
	A graph with radius $r$ has diameter $d \le 2r$.
	By Theorem $3.1$ of \cite{SC19} and using $d \le 2r$, we have an upper bound of the desired form.
	On the other hand, the construction in Section $3$ of \cite{SC19} for $d$ being even is a rooted tree of radius $r$ with Wiener index of the desired form. 
	Taking a center $c$ of any graph $G$ with radius $r$, we can construct a spanning tree $T$ with a breadth-first search algorithm which has radius $r$ as well. If $G$ was not a tree, then $W(G)<W(T)$, so we conclude extremal graphs are trees.
\end{proof}

In the following two subsections, we derive results in the digraph case.
Here we only consider biconnected digraphs, as otherwise the total distance can be infinite.

\subsection{in terms of order and radius}

\begin{thr}\label{maxmurad}
	When $D$ is a digraph with radius $r$, then the maximal possible Wiener index of a digraph with order $n$ and radius $r$ is of the form $2rn^2-4r^2n+\Theta_r(1).$
	For $n>32r^3$, the extremal digraphs contain $DP_n^{2r}$, a directed cycle of length $d=2r$ with a blow-up of one vertex by an independent set of size $n-d+1,$ as a subdigraph and for small $r$ the extremal digraphs are known.
\end{thr}

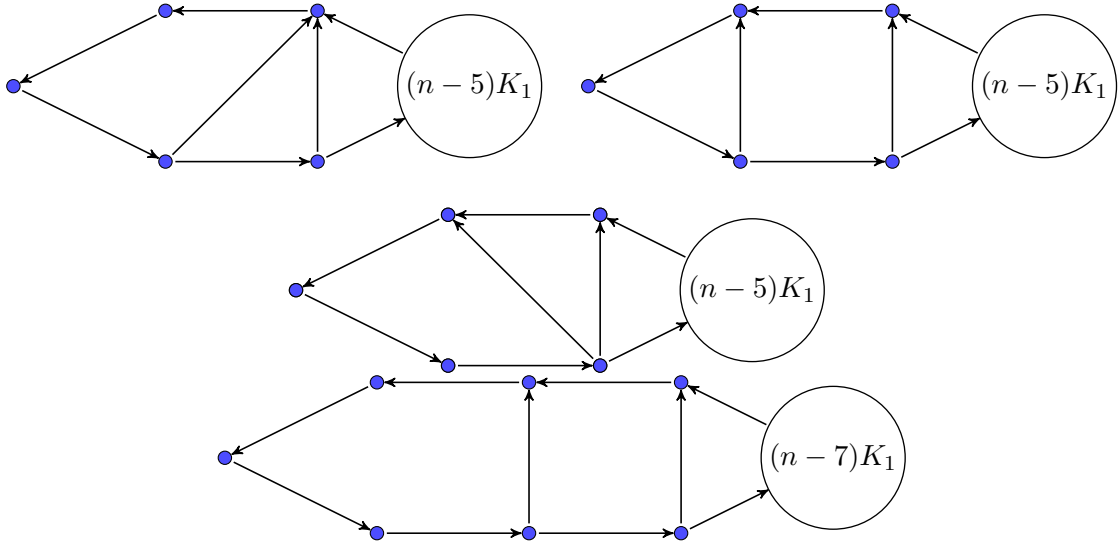
\begin{figure}[h]
	\centering
	\begin{tikzpicture}
	\definecolor{cv0}{rgb}{0.0,0.0,0.0}
	\definecolor{c}{rgb}{1.0,1.0,1.0}
	
	%
	\Vertex[L=\hbox{$(n-5)K_1$},x=6cm,y=0.0cm]{v5}
	
\node[label=center:\(\),fill=ududff,style={minimum size=5,shape=circle}] (v0) at (4,1) {};
\node[label=center:\(\),fill=ududff,style={minimum size=5}] (v1) at (2,1) {};
\node[label=center:\(\),fill=ududff,style={minimum size=5,shape=circle}] (v2) at (0,0) {};
\node[label=center:\(\),fill=ududff,style={minimum size=5,shape=circle}] (v3) at (2,-1) {};
\node[label=center:\(\),fill=ududff,style={minimum size=5,shape=circle}] (v4) at (4,-1) {};
	
	\Edge[lw=0.1cm,style={post, right}](v0)(v1)
	\Edge[lw=0.1cm,style={post, right}](v1)(v2)
	\Edge[lw=0.1cm,style={post, right}](v2)(v3)
	\Edge[lw=0.1cm,style={post, right}](v3)(v0)
	\Edge[lw=0.1cm,style={post, right}](v3)(v4)
	\Edge[lw=0.1cm,style={post, right}](v4)(v0)
	\Edge[lw=0.1cm,style={post, right}](v4)(v5)
	\Edge[lw=0.1cm,style={post, right}](v5)(v0)
	\end{tikzpicture}
	\quad
	\centering
	\begin{tikzpicture}
	\definecolor{cv0}{rgb}{0.0,0.0,0.0}
	\definecolor{c}{rgb}{1.0,1.0,1.0}

\node[label=center:\(\),fill=ududff,style={minimum size=5,shape=circle}] (v0) at (4,1) {};
\node[label=center:\(\),fill=ududff,style={minimum size=5}] (v1) at (2,1) {};
\node[label=center:\(\),fill=ududff,style={minimum size=5,shape=circle}] (v2) at (0,0) {};
\node[label=center:\(\),fill=ududff,style={minimum size=5,shape=circle}] (v3) at (2,-1) {};
\node[label=center:\(\),fill=ududff,style={minimum size=5,shape=circle}] (v4) at (4,-1) {};
	\Vertex[L=\hbox{$(n-5)K_1$},x=6cm,y=0.0cm]{v5}
	
	\Edge[lw=0.1cm,style={post, right}](v0)(v1)
	\Edge[lw=0.1cm,style={post, right}](v1)(v2)
	\Edge[lw=0.1cm,style={post, right}](v2)(v3)
	\Edge[lw=0.1cm,style={post, right}](v3)(v1)
	\Edge[lw=0.1cm,style={post, right}](v3)(v4)
	\Edge[lw=0.1cm,style={post, right}](v4)(v0)
	\Edge[lw=0.1cm,style={post, right}](v4)(v5)
	\Edge[lw=0.1cm,style={post, right}](v5)(v0)
	\end{tikzpicture}
	
	\quad
	\centering	
	
	\begin{tikzpicture}
\node[label=center:\(\),fill=ududff,style={minimum size=5,shape=circle}] (v0) at (4,1) {};
\node[label=center:\(\),fill=ududff,style={minimum size=5}] (v1) at (2,1) {};
\node[label=center:\(\),fill=ududff,style={minimum size=5,shape=circle}] (v2) at (0,0) {};
\node[label=center:\(\),fill=ududff,style={minimum size=5,shape=circle}] (v3) at (2,-1) {};
\node[label=center:\(\),fill=ududff,style={minimum size=5,shape=circle}] (v4) at (4,-1) {};

\node[label=center:\(\),fill=ududff,style={minimum size=5,shape=circle}] (v0) at (4,1) {};
\node[label=center:\(\),fill=ududff,style={minimum size=5}] (v1) at (2,1) {};
\node[label=center:\(\),fill=ududff,style={minimum size=5,shape=circle}] (v2) at (0,0) {};
\node[label=center:\(\),fill=ududff,style={minimum size=5,shape=circle}] (v3) at (2,-1) {};
\node[label=center:\(\),fill=ududff,style={minimum size=5,shape=circle}] (v4) at (4,-1) {};
	\Vertex[L=\hbox{$(n-5)K_1$},x=6cm,y=0.0cm]{v5}
	
	\Edge[lw=0.1cm,style={post, right}](v0)(v1)
	\Edge[lw=0.1cm,style={post, right}](v1)(v2)
	\Edge[lw=0.1cm,style={post, right}](v2)(v3)
	\Edge[lw=0.1cm,style={post, right}](v4)(v1)
	\Edge[lw=0.1cm,style={post, right}](v3)(v4)
	\Edge[lw=0.1cm,style={post, right}](v4)(v0)
	\Edge[lw=0.1cm,style={post, right}](v4)(v5)
	\Edge[lw=0.1cm,style={post, right}](v5)(v0)
	\end{tikzpicture}
	\quad
	\centering

	\begin{tikzpicture}
	\definecolor{cv0}{rgb}{0.0,0.0,0.0}
	\definecolor{c}{rgb}{1.0,1.0,1.0}

	\node[label=center:\(\),fill=ududff,style={minimum size=5}] (v1) at (6,1) {};
	\node[label=center:\(\),fill=ududff,style={minimum size=5}] (v2) at (4,1) {};
	\node[label=center:\(\),fill=ududff,style={minimum size=5}] (v3) at (2,1) {};
	\node[label=center:\(\),fill=ududff,style={minimum size=5}] (v4) at (0,0) {};
	\node[label=center:\(\),fill=ududff,style={minimum size=5}] (v5) at (2,-1) {};
	\node[label=center:\(\),fill=ududff,style={minimum size=5}] (v6) at (4,-1) {};
	\node[label=center:\(\),fill=ududff,style={minimum size=5}] (v7) at (6,-1) {};

	\Vertex[L=\hbox{$(n-7)K_1$},x=8cm,y=0.0cm]{v0}
	
	\Edge[lw=0.1cm,style={post, right}](v0)(v1)
	\Edge[lw=0.1cm,style={post, right}](v1)(v2)
	\Edge[lw=0.1cm,style={post, right}](v2)(v3)
	\Edge[lw=0.1cm,style={post, right}](v3)(v4)
	\Edge[lw=0.1cm,style={post, right}](v4)(v5)
	\Edge[lw=0.1cm,style={post, right}](v5)(v6)
	\Edge[lw=0.1cm,style={post, right}](v6)(v7)
	
	\Edge[lw=0.1cm,style={post, right}](v7)(v0)
	
	\Edge[lw=0.1cm,style={post, right}](v7)(v1)
	\Edge[lw=0.1cm,style={post, right}](v6)(v2)

	\end{tikzpicture}
	\caption{Extremal digraphs with $\rad=r\in \{3,4\}$ and maximum Wiener index}
	\label{fig:digraph_rad_maxW_34}
\end{figure}

\begin{proof}
	Take a center $x$ such that $d(x,V)+d(V,x)=2r$. 
	Then for any $2$ vertices $u,v$, we have $d(u,v) \le d(u,x)+d(x,v) \le 2r$ by the triangle inequality and the definition of $x$. Hence the diameter of the digraph is at most $2r$
	and thus $W(D) \le 2rn^2-4r^2n+O(r^3)$ due to Theorem $5.1$ of \cite{SC19}.
	Let $D'$ be the digraph formed from adding the two additional directed edges $\vec{u_r u_1}$ and $\vec{u_{2r-1}u_r}$ to $DP_n^{2r}$.
	Then $\rad(D')=r$ and \begin{align*}
	W(D')&\ge 2r(n-2r+1)(n-2r)+2r(2r-1)(n-2r+1)+2W(C_{r})\\&
	=2rn^2-4r^2n+2r(2r-1)+r^3-r^2\\&=2rn^2-4r^2n+\Omega(r^3), 
	\end{align*}
	from which the first part of the theorem follows.
	
	Let $d=2r$ and define $$\Delta_u=\sum_{v \in V: v \not=u} \left( 2d-d(u,v)-d(v,u) \right).$$
	Analogously to the proof of Theorem $5.1$ in \cite{SC19}, we find for $n \ge 4d^3$ that there exist two vertices $u$ and $v$ with $\Delta_v=\Delta_u=d^2-d$ and $d(u,v)=d(v,u)=2r$ and given the shortest path $u u_1 u_2 u_3 \ldots u_{d-1} v$, we again see that for any other vertex $w$ there are the directed edges $\vec{wu_1}$ and $\vec{u_{2r-1}w}$, as well as the directed edges $\vec{u_{2r-1}u}$ and $\vec{vu_1}.$
	
	For $r=1$, the extremal digraphs are exactly $DP_n^2.$
	For $r\ge 2$, we need to add some edges $\vec{u_i u_j}$ where $i>j$ such that $d(u_r,V)=d(V,u_r)=r.$
	For $r=2$, there are $3$ possible directed edges to add and so it is easy to see that we only need to add $\vec{u_3 u_1}$ to get the extremal digraph.
	Using a computer program, we found the optimal ways to do this to get the maximal Wiener index for $r\le 4$.
	They are presented in Figure~\ref{fig:digraph_rad_maxW_34}.			
	\end{proof}

For small $r$, the extremal digraphs are formed by adding some small number $k$ directed edges to $DP_n^{2r}$ between the $2r-1$ vertices from the cycle different from the vertex which is used in the blow-up process. 
For $5 \le r \le 7$, we used a computer program to find the constructions which achieve the first local maximum of the total distance as a function of $k$.
We conjecture these digraphs to be extremal for large $n$ and present them in Figure~\ref{fig:digraph_rad_maxW}.

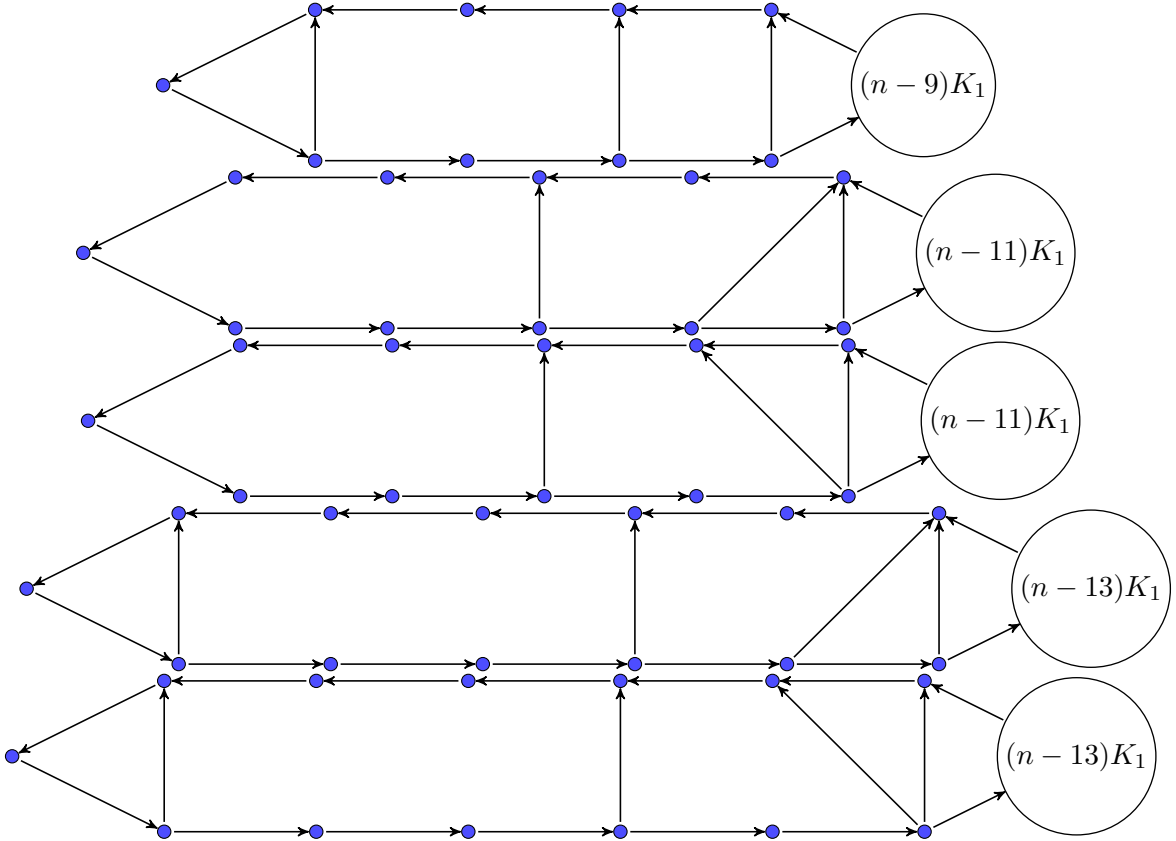
\begin{figure}[h]
\centering

\begin{tikzpicture}
\definecolor{cv0}{rgb}{0.0,0.0,0.0}
\definecolor{c}{rgb}{1.0,1.0,1.0}

\node[label=center:\(\),fill=ududff,style={minimum size=5}] (v1) at (8,1) {};
\node[label=center:\(\),fill=ududff,style={minimum size=5}] (v2) at (6,1) {};
\node[label=center:\(\),fill=ududff,style={minimum size=5}] (v3) at (4,1) {};
\node[label=center:\(\),fill=ududff,style={minimum size=5}] (v4) at (2,1) {};
\node[label=center:\(\),fill=ududff,style={minimum size=5}] (v5) at (0,0) {};
\node[label=center:\(\),fill=ududff,style={minimum size=5}] (v6) at (2,-1) {};
\node[label=center:\(\),fill=ududff,style={minimum size=5}] (v7) at (4,-1) {};
\node[label=center:\(\),fill=ududff,style={minimum size=5}] (v8) at (6,-1) {};
\node[label=center:\(\),fill=ududff,style={minimum size=5}] (v9) at (8,-1) {};


\Vertex[L=\hbox{$(n-9)K_1$},x=10cm,y=0.0cm]{v0}

\Edge[lw=0.1cm,style={post, right}](v0)(v1)
\Edge[lw=0.1cm,style={post, right}](v1)(v2)
\Edge[lw=0.1cm,style={post, right}](v2)(v3)
\Edge[lw=0.1cm,style={post, right}](v3)(v4)
\Edge[lw=0.1cm,style={post, right}](v4)(v5)
\Edge[lw=0.1cm,style={post, right}](v5)(v6)
\Edge[lw=0.1cm,style={post, right}](v6)(v7)
\Edge[lw=0.1cm,style={post, right}](v7)(v8)
\Edge[lw=0.1cm,style={post, right}](v8)(v9)

\Edge[lw=0.1cm,style={post, right}](v9)(v0)

\Edge[lw=0.1cm,style={post, right}](v9)(v1)
\Edge[lw=0.1cm,style={post, right}](v8)(v2)
\Edge[lw=0.1cm,style={post, right}](v6)(v4)

\end{tikzpicture}
\quad
\centering

\begin{tikzpicture}
\definecolor{cv0}{rgb}{0.0,0.0,0.0}
\definecolor{c}{rgb}{1.0,1.0,1.0}

\node[label=center:\(\),fill=ududff,style={minimum size=5}] (v1) at (10,1) {};
\node[label=center:\(\),fill=ududff,style={minimum size=5}] (v2) at (8,1) {};
\node[label=center:\(\),fill=ududff,style={minimum size=5}] (v3) at (6,1) {};
\node[label=center:\(\),fill=ududff,style={minimum size=5}] (v4) at (4,1) {};
\node[label=center:\(\),fill=ududff,style={minimum size=5}] (v5) at (2,1) {};
\node[label=center:\(\),fill=ududff,style={minimum size=5}] (v6) at (0,0) {};
\node[label=center:\(\),fill=ududff,style={minimum size=5}] (v7) at (2,-1) {};
\node[label=center:\(\),fill=ududff,style={minimum size=5}] (v8) at (4,-1) {};
\node[label=center:\(\),fill=ududff,style={minimum size=5}] (v9) at (6,-1) {};
\node[label=center:\(\),fill=ududff,style={minimum size=5}] (v10) at (8,-1) {};
\node[label=center:\(\),fill=ududff,style={minimum size=5}] (v11) at (10,-1) {};

\Vertex[L=\hbox{$(n-11)K_1$},x=12cm,y=0.0cm]{v0}

\Edge[lw=0.1cm,style={post, right}](v0)(v1)
\Edge[lw=0.1cm,style={post, right}](v1)(v2)
\Edge[lw=0.1cm,style={post, right}](v2)(v3)
\Edge[lw=0.1cm,style={post, right}](v3)(v4)
\Edge[lw=0.1cm,style={post, right}](v4)(v5)
\Edge[lw=0.1cm,style={post, right}](v5)(v6)
\Edge[lw=0.1cm,style={post, right}](v6)(v7)
\Edge[lw=0.1cm,style={post, right}](v7)(v8)
\Edge[lw=0.1cm,style={post, right}](v8)(v9)
\Edge[lw=0.1cm,style={post, right}](v9)(v10)
\Edge[lw=0.1cm,style={post, right}](v10)(v11)

\Edge[lw=0.1cm,style={post, right}](v11)(v0)

\Edge[lw=0.1cm,style={post, right}](v11)(v1)
\Edge[lw=0.1cm,style={post, right}](v9)(v3)

\Edge[lw=0.1cm,style={post, right}](v10)(v1)

\end{tikzpicture}
\quad
\centering

\begin{tikzpicture}

\node[label=center:\(\),fill=ududff,style={minimum size=5}] (v1) at (10,1) {};
\node[label=center:\(\),fill=ududff,style={minimum size=5}] (v2) at (8,1) {};
\node[label=center:\(\),fill=ududff,style={minimum size=5}] (v3) at (6,1) {};
\node[label=center:\(\),fill=ududff,style={minimum size=5}] (v4) at (4,1) {};
\node[label=center:\(\),fill=ududff,style={minimum size=5}] (v5) at (2,1) {};
\node[label=center:\(\),fill=ududff,style={minimum size=5}] (v6) at (0,0) {};
\node[label=center:\(\),fill=ududff,style={minimum size=5}] (v7) at (2,-1) {};
\node[label=center:\(\),fill=ududff,style={minimum size=5}] (v8) at (4,-1) {};
\node[label=center:\(\),fill=ududff,style={minimum size=5}] (v9) at (6,-1) {};
\node[label=center:\(\),fill=ududff,style={minimum size=5}] (v10) at (8,-1) {};
\node[label=center:\(\),fill=ududff,style={minimum size=5}] (v11) at (10,-1) {};


\Vertex[L=\hbox{$(n-11)K_1$},x=12cm,y=0.0cm]{v0}

\Edge[lw=0.1cm,style={post, right}](v0)(v1)
\Edge[lw=0.1cm,style={post, right}](v1)(v2)
\Edge[lw=0.1cm,style={post, right}](v2)(v3)
\Edge[lw=0.1cm,style={post, right}](v3)(v4)
\Edge[lw=0.1cm,style={post, right}](v4)(v5)
\Edge[lw=0.1cm,style={post, right}](v5)(v6)
\Edge[lw=0.1cm,style={post, right}](v6)(v7)
\Edge[lw=0.1cm,style={post, right}](v7)(v8)
\Edge[lw=0.1cm,style={post, right}](v8)(v9)
\Edge[lw=0.1cm,style={post, right}](v9)(v10)
\Edge[lw=0.1cm,style={post, right}](v10)(v11)

\Edge[lw=0.1cm,style={post, right}](v11)(v0)

\Edge[lw=0.1cm,style={post, right}](v11)(v1)
\Edge[lw=0.1cm,style={post, right}](v9)(v3)

\Edge[lw=0.1cm,style={post, right}](v11)(v2)

\end{tikzpicture}

\quad
\centering	
\begin{tikzpicture}
\definecolor{cv0}{rgb}{0.0,0.0,0.0}
\definecolor{c}{rgb}{1.0,1.0,1.0}
\node[label=center:\(\),fill=ududff,style={minimum size=5}] (v1) at (12,1) {};
\node[label=center:\(\),fill=ududff,style={minimum size=5}] (v2) at (10,1) {};
\node[label=center:\(\),fill=ududff,style={minimum size=5}] (v3) at (8,1) {};
\node[label=center:\(\),fill=ududff,style={minimum size=5}] (v4) at (6,1) {};
\node[label=center:\(\),fill=ududff,style={minimum size=5}] (v5) at (4,1) {};
\node[label=center:\(\),fill=ududff,style={minimum size=5}] (v6) at (2,1) {};
\node[label=center:\(\),fill=ududff,style={minimum size=5}] (v7) at (0,0) {};
\node[label=center:\(\),fill=ududff,style={minimum size=5}] (v8) at (2,-1) {};
\node[label=center:\(\),fill=ududff,style={minimum size=5}] (v9) at (4,-1) {};
\node[label=center:\(\),fill=ududff,style={minimum size=5}] (v10) at (6,-1) {};
\node[label=center:\(\),fill=ududff,style={minimum size=5}] (v11) at (8,-1) {};
\node[label=center:\(\),fill=ududff,style={minimum size=5}] (v12) at (10,-1) {};
\node[label=center:\(\),fill=ududff,style={minimum size=5}] (v13) at (12,-1) {};

\Vertex[L=\hbox{$(n-13)K_1$},x=14cm,y=0.0cm]{v0}

\Edge[lw=0.1cm,style={post, right}](v0)(v1)
\Edge[lw=0.1cm,style={post, right}](v1)(v2)
\Edge[lw=0.1cm,style={post, right}](v2)(v3)
\Edge[lw=0.1cm,style={post, right}](v3)(v4)
\Edge[lw=0.1cm,style={post, right}](v4)(v5)
\Edge[lw=0.1cm,style={post, right}](v5)(v6)
\Edge[lw=0.1cm,style={post, right}](v6)(v7)
\Edge[lw=0.1cm,style={post, right}](v7)(v8)
\Edge[lw=0.1cm,style={post, right}](v8)(v9)
\Edge[lw=0.1cm,style={post, right}](v9)(v10)
\Edge[lw=0.1cm,style={post, right}](v10)(v11)
\Edge[lw=0.1cm,style={post, right}](v11)(v12)
\Edge[lw=0.1cm,style={post, right}](v12)(v13)

\Edge[lw=0.1cm,style={post, right}](v13)(v0)

\Edge[lw=0.1cm,style={post, right}](v13)(v1)
\Edge[lw=0.1cm,style={post, right}](v11)(v3)
\Edge[lw=0.1cm,style={post, right}](v8)(v6)

\Edge[lw=0.1cm,style={post, right}](v12)(v1)

\end{tikzpicture}
\quad
\centering	
\begin{tikzpicture}
\definecolor{cv0}{rgb}{0.0,0.0,0.0}
\definecolor{c}{rgb}{1.0,1.0,1.0}

\node[label=center:\(\),fill=ududff,style={minimum size=5}] (v1) at (12,1) {};
\node[label=center:\(\),fill=ududff,style={minimum size=5}] (v2) at (10,1) {};
\node[label=center:\(\),fill=ududff,style={minimum size=5}] (v3) at (8,1) {};
\node[label=center:\(\),fill=ududff,style={minimum size=5}] (v4) at (6,1) {};
\node[label=center:\(\),fill=ududff,style={minimum size=5}] (v5) at (4,1) {};
\node[label=center:\(\),fill=ududff,style={minimum size=5}] (v6) at (2,1) {};
\node[label=center:\(\),fill=ududff,style={minimum size=5}] (v7) at (0,0) {};
\node[label=center:\(\),fill=ududff,style={minimum size=5}] (v8) at (2,-1) {};
\node[label=center:\(\),fill=ududff,style={minimum size=5}] (v9) at (4,-1) {};
\node[label=center:\(\),fill=ududff,style={minimum size=5}] (v10) at (6,-1) {};
\node[label=center:\(\),fill=ududff,style={minimum size=5}] (v11) at (8,-1) {};
\node[label=center:\(\),fill=ududff,style={minimum size=5}] (v12) at (10,-1) {};
\node[label=center:\(\),fill=ududff,style={minimum size=5}] (v13) at (12,-1) {};

%

\Vertex[L=\hbox{$(n-13)K_1$},x=14cm,y=0.0cm]{v0}

\Edge[lw=0.1cm,style={post, right}](v0)(v1)
\Edge[lw=0.1cm,style={post, right}](v1)(v2)
\Edge[lw=0.1cm,style={post, right}](v2)(v3)
\Edge[lw=0.1cm,style={post, right}](v3)(v4)
\Edge[lw=0.1cm,style={post, right}](v4)(v5)
\Edge[lw=0.1cm,style={post, right}](v5)(v6)
\Edge[lw=0.1cm,style={post, right}](v6)(v7)
\Edge[lw=0.1cm,style={post, right}](v7)(v8)
\Edge[lw=0.1cm,style={post, right}](v8)(v9)
\Edge[lw=0.1cm,style={post, right}](v9)(v10)
\Edge[lw=0.1cm,style={post, right}](v10)(v11)
\Edge[lw=0.1cm,style={post, right}](v11)(v12)
\Edge[lw=0.1cm,style={post, right}](v12)(v13)

\Edge[lw=0.1cm,style={post, right}](v13)(v0)

\Edge[lw=0.1cm,style={post, right}](v13)(v1)
\Edge[lw=0.1cm,style={post, right}](v11)(v3)
\Edge[lw=0.1cm,style={post, right}](v8)(v6)

\Edge[lw=0.1cm,style={post, right}](v13)(v2)

\end{tikzpicture}

\caption{Conjectured extremal digraphs with $\rad=r\in \{5,6,7\}$ and maximum Wiener index}
\label{fig:digraph_rad_maxW}
\end{figure}

\subsection{in terms of order and in- or out-radius}

We first note that the extremal value will be the same in both cases. Changing all direction of the edges of a digraph $D$ with in-radius $r$ results in a digraph $D'$ with out-radius $r$ and vice versa, since $d_D(u,v)=d_{D'}(v,u)$ for every two vertices, the Wiener indices $W(D)$ and $W(D')$ will be the same.
Hence it suffices to prove this question for the out-radius $\rad^+$.
We first prove an asymptotic result.

\begin{thr}\label{maxmurad+}
	Given a digraph $D$ of order $n$ with $\rad^+=r>1$. Then the Wiener index is at most $ \frac{n^3}{3} + (r-1)\Theta(n^2) $.

\end{thr}

\begin{proof}
	Analogous to the proof of Theorem~\ref{rad+=1}, we find that $W(D) \le \frac{n^3-n}{3}+(r-1)(n^2-n).$
	For this, note that $d(v,x)$ has been increased by at most $r-1$ and so has every term $d(u,x)$.

	To prove the bound of this theorem is sharp, take the digraph $D$ in Figure~\ref{fig:digraph_rad+r} which has out-radius $r$, where $n=qr+k$ with $2 \le k \le r+1.$ 
	For every pair of vertices $v_{pr+i}$ and $u$ (where $1 \le i \le r$ and  $u=v_j$ for some $j>pr+i$ or $u=v$), we have $d(v_{pr+i}, u )+d(u, v_{pr+i})=n-pr$.
	The Wiener index of this graph equals

	\begin{align*} 
	W(D) &=\sum_{p=0}^{q-1} \sum_{i=1}^r (n-pr)(n-pr-i) + k \frac{k(k-1)}2 \\
	&=\frac{n^3}{3}+\frac{r-1}4 n^2 -\frac{r(r+3)}{12}n+O_r(1). \qedhere
	\end{align*}
\end{proof}

\begin{figure}[h]
	\centering

	\begin{tikzpicture}[line cap=round,line join=round,>= {Latex[length=3mm, width=1.5mm]} ,x=1.3cm,y=1.3cm]
	\clip(-1.3,-0.5) rectangle (10,2.5);
	\draw [->,line width=1.1pt] (5.36,2) -- (0.,0.);
	\draw [->,line width=1.1pt] (5.36,2) -- (4.,0.);
	\draw [->,line width=1.1pt] (5.36,2) -- (7.5,0.);
	\draw [<-,line width=1.1pt] (5.36,2) --(9.5,2)-- (9.5,0.);
	\draw [->,line width=1.1pt] (0.,0.) -- (1.,0.);
	\draw [->,line width=1.1pt] (1.,0.) -- (2.,0.);
	\draw [line width=1.pt,dotted] (2.,0.) -- (3.,0.);
	\draw [->,line width=1.1pt] (3.,0.) -- (4.,0.);
	\draw [->,line width=1.1pt] (4.,0.) -- (5.,0.);
	\draw [line width=1.pt,dotted] (5.,0.)-- (6.5,0.);
	\draw [->,line width=1.1pt] (6.5,0.) -- (7.5,0.);
	\draw [line width=1.pt,dotted] (8.5,0.)-- (9.5,0.);
	\draw [->,line width=1.1pt] (7.5,0.) -- (8.5,0.);
	\begin{scriptsize}
	\draw [fill=ududff] (0.,0.) circle (2.5pt);
	\draw [fill=ududff] (5.36,2) circle (2.5pt);
	\draw [fill=xdxdff] (1.,0.) circle (2.5pt);
	\draw [fill=xdxdff] (2.,0.) circle (2.5pt);
	\draw [fill=ududff] (3.,0.) circle (2.5pt);
	\draw [fill=xdxdff] (4.,0.) circle (2.5pt);
	\draw [fill=xdxdff] (5.,0.) circle (2.5pt);
	\draw [fill=xdxdff] (6.5,0.) circle (2.5pt);
	\draw [fill=xdxdff] (7.5,0.) circle (2.5pt);
	\draw [fill=xdxdff] (8.5,0.) circle (2.5pt);
	\draw [fill=xdxdff] (9.5,0.) circle (2.5pt);
	
	\coordinate [label=center:$v$] (A) at (5.36,2.2); 
	\coordinate [label=center:$v_{n-1}$] (A) at (9.5,-0.2); 
	\coordinate [label=center:$v_{qr+2}$] (A) at (8.5,-0.2); 
	\coordinate [label=center:$v_{qr+1}$] (A) at (7.5,-0.20); 
	\coordinate [label=center:$v_{qr}$] (A) at (6.5,-0.2); 
	
	\coordinate [label=center:$v_{r}$] (A) at (3,-0.20); 
	\coordinate [label=center:$v_{3}$] (A) at (2,-0.20); 
	\coordinate [label=center:$v_{2}$] (A) at (1,-0.20); 
	\coordinate [label=center:$v_{1}$] (A) at (0,-0.20);
	\coordinate [label=center:$v_{r+2}$] (A) at (5,-0.20); 
	\coordinate [label=center:$v_{r+1}$] (A) at (4,-0.20);  
	
	\end{scriptsize}
	\end{tikzpicture}

	\caption{Digraph with $\rad^{+}=r$ and large Wiener index}
	\label{fig:digraph_rad+r}
\end{figure}
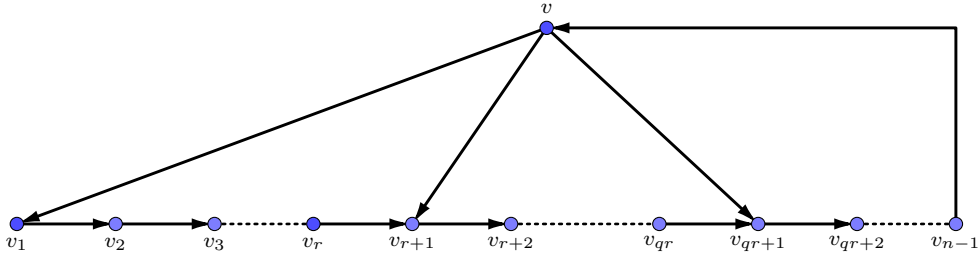

When the out-radius equals one, the problem can be solved exactly.

\begin{thr}\label{rad+=1}
	Given a digraph $D$ of order $n$ with outradius equal to $1$, the Wiener index of $D$ is at most $\frac{n^3-n}3 $ with equality if and only if the graph is isomorphic to one of the two configurations given in Figure~\ref{fig:digraph_rad+}.
\end{thr}

\begin{proof}
	By definition of $\rad^+$, there is some vertex $v$ such that for every vertex $u \in V(D)$ different from $v$, there is a directed edge from $v$ towards $u$.
	There is a shortest path from every $u$ to $v$, having some length $i$.
	We denote the number of vertices $u$ with $d(u,v)=i$ with $X_i$.
	Call this path $u v_{i-1}v_{i-2} \ldots v_1 v$.
	Then $d(u,v_{j})=i-j$, while for every vertex $x$ different from any such $v_j$ and $v,u$, we have $d(u,x) \le d(u,v)+d(v,x)=i+1$.
	Hence $\sum_{x \in V(G)} d(u,x) \le \sum_{j=1}^{i} j + (n-i-1)(i+1)=
	\frac12 (2n-i-2)(i+1)$.
	Note that the parabolic function $f(i)=\frac 12 (2n-i-2)(i+1)$ is an increasing function in $i$ up to $\frac{2n-3}{2}$ and $f$ obtains the same values in $n-1$ and $n-2$.
	Since $i$ is at most $n-1$, we find that 
	\begin{align*}
	W(G) &\le \sum_{i=1}^{n-1}  X_i f(i) + (n-1)\\
	&\le \sum_{i=1}^{n-1} f(i) + (n-1)\\
	&= \frac{n^3-n}3
	\end{align*}
	with equality if and only if $X_i=1$ for every $i \le n-3$, from which one can conclude that there are only the two given cases of equality.	
\end{proof}

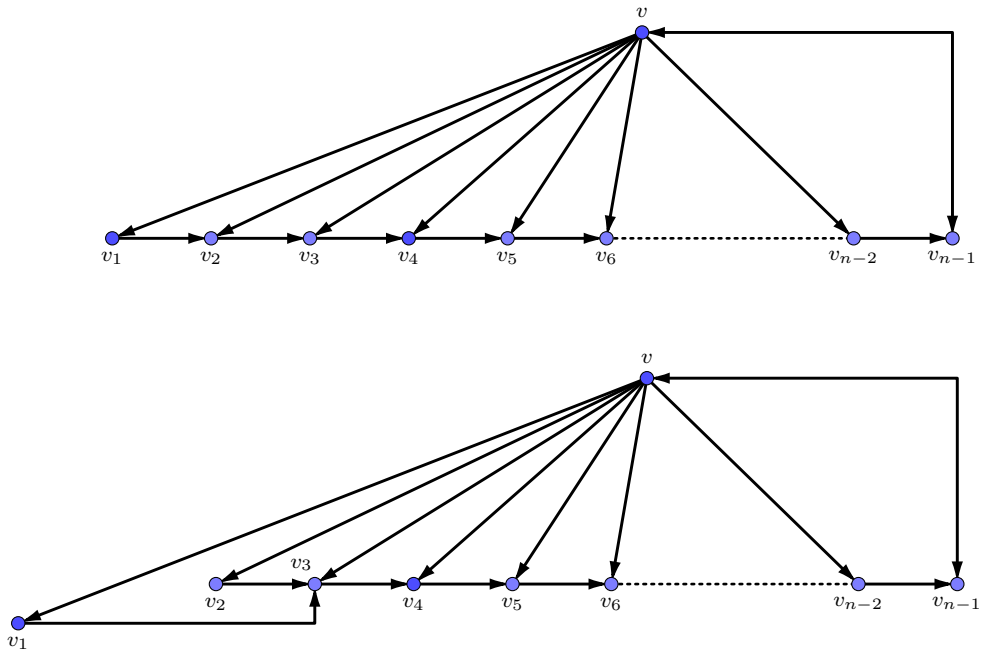
\begin{figure}[h]
	\centering

\begin{tikzpicture}[line cap=round,line join=round,>= {Latex[length=3mm, width=1.5mm]} ,x=1.3cm,y=1.3cm]
\clip(-1.3,-1) rectangle (9,2.5);
\draw [->,line width=1.1pt] (5.36,2.1) -- (0.,0.);
\draw [->,line width=1.1pt] (5.36,2.1) -- (1.,0.);
\draw [->,line width=1.1pt] (5.36,2.1) -- (2.,0.);
\draw [->,line width=1.1pt] (5.36,2.1) -- (3.,0.);
\draw [->,line width=1.1pt] (5.36,2.1) -- (4.,0.);
\draw [->,line width=1.1pt] (5.36,2.1) -- (5.,0.);
\draw [->,line width=1.1pt] (5.36,2.1) -- (7.5,0.);
\draw [<->,line width=1.1pt] (5.36,2.1) --  (8.5,2.1) -- (8.5,0.);
\draw [->,line width=1.1pt] (0.,0.) -- (1.,0.);
\draw [->,line width=1.1pt] (1.,0.) -- (2.,0.);
\draw [->,line width=1.1pt] (2.,0.) -- (3.,0.);
\draw [->,line width=1.1pt] (3.,0.) -- (4.,0.);
\draw [->,line width=1.1pt] (4.,0.) -- (5.,0.);
\draw [line width=1.pt,dotted] (5.,0.)-- (7.5,0.);
\draw [->,line width=1.1pt] (7.5,0.) -- (8.5,0.);
\begin{scriptsize}
\draw [fill=ududff] (0.,0.) circle (2.5pt);
\draw [fill=ududff] (5.36,2.1) circle (2.5pt);
\draw [fill=xdxdff] (1.,0.) circle (2.5pt);
\draw [fill=xdxdff] (2.,0.) circle (2.5pt);
\draw [fill=ududff] (3.,0.) circle (2.5pt);
\draw [fill=xdxdff] (4.,0.) circle (2.5pt);
\draw [fill=xdxdff] (5.,0.) circle (2.5pt);
\draw [fill=xdxdff] (7.5,0.) circle (2.5pt);
\draw [fill=xdxdff] (8.5,0.) circle (2.5pt);

\coordinate [label=center:$v$] (A) at (5.36,2.3); 
\coordinate [label=center:$v_{n-1}$] (A) at (8.5,-0.2); 
\coordinate [label=center:$v_{n-2}$] (A) at (7.5,-0.20); 

\coordinate [label=center:$v_{4}$] (A) at (3,-0.20); 
\coordinate [label=center:$v_{3}$] (A) at (2,-0.20); 
\coordinate [label=center:$v_{2}$] (A) at (1,-0.20); 
\coordinate [label=center:$v_{1}$] (A) at (0,-0.20);
\coordinate [label=center:$v_{6}$] (A) at (5,-0.20); 
\coordinate [label=center:$v_{5}$] (A) at (4,-0.20);  

\end{scriptsize}
\end{tikzpicture}	\quad
\centering	

\begin{tikzpicture}[line cap=round,line join=round,>= {Latex[length=3mm, width=1.5mm]} ,x=1.3cm,y=1.3cm]
\clip(-1.3,-0.7) rectangle (9,2.5);
\draw [->,line width=1.1pt] (5.36,2.1) -- (-1,-0.4);
\draw [->,line width=1.1pt] (5.36,2.1) -- (1.,0.);
\draw [->,line width=1.1pt] (5.36,2.1) -- (2.,0.);
\draw [->,line width=1.1pt] (5.36,2.1) -- (3.,0.);
\draw [->,line width=1.1pt] (5.36,2.1) -- (4.,0.);
\draw [->,line width=1.1pt] (5.36,2.1) -- (5.,0.);
\draw [->,line width=1.1pt] (5.36,2.1) -- (7.5,0.);
\draw [<->,line width=1.1pt] (5.36,2.1) --  (8.5,2.1) -- (8.5,0.);
\draw [->,line width=1.1pt] (-1,-0.4) -- (2,-0.4)-- (2.,0.);
\draw [->,line width=1.1pt] (1.,0.) -- (2.,0.);
\draw [->,line width=1.1pt] (2.,0.) -- (3.,0.);
\draw [->,line width=1.1pt] (3.,0.) -- (4.,0.);
\draw [->,line width=1.1pt] (4.,0.) -- (5.,0.);
\draw [line width=1.pt,dotted] (5.,0.)-- (7.5,0.);
\draw [->,line width=1.1pt] (7.5,0.) -- (8.5,0.);
\begin{scriptsize}
\draw [fill=ududff] (-1,-0.4) circle (2.5pt);
\draw [fill=ududff] (5.36,2.1) circle (2.5pt);
\draw [fill=xdxdff] (1.,0.) circle (2.5pt);
\draw [fill=xdxdff] (2.,0.) circle (2.5pt);
\draw [fill=ududff] (3.,0.) circle (2.5pt);
\draw [fill=xdxdff] (4.,0.) circle (2.5pt);
\draw [fill=xdxdff] (5.,0.) circle (2.5pt);
\draw [fill=xdxdff] (7.5,0.) circle (2.5pt);
\draw [fill=xdxdff] (8.5,0.) circle (2.5pt);

\coordinate [label=center:$v$] (A) at (5.36,2.3); 
\coordinate [label=center:$v_{n-1}$] (A) at (8.5,-0.2); 
\coordinate [label=center:$v_{n-2}$] (A) at (7.5,-0.20); 
\coordinate [label=center:$v_{6}$] (A) at (5,-0.20); 
\coordinate [label=center:$v_{5}$] (A) at (4,-0.20); 
\coordinate [label=center:$v_{4}$] (A) at (3,-0.20); 
\coordinate [label=left:$v_{3}$] (A) at (2,0.20); 
\coordinate [label=center:$v_{2}$] (A) at (1,-0.20); 
\coordinate [label=center:$v_{1}$] (A) at (-1,-0.6); 

\end{scriptsize}
\end{tikzpicture}
\caption{The two extremal digraphs with outradius $1$ and maximal Wiener index}
	\label{fig:digraph_rad+}
\end{figure}
\section{Conclusion}\label{conc}
	The question of determining the minimum total distance among all graphs or digraphs of order $n$ and (out)radius $r$ and characterizing the extremal (di)graphs has been solved for $n$ large enough compared with $r$.
	This question has been solved only asymptotically in the digraph case using the radius.
	In both the graph- and digraph case, it might be challenging to solve the question completely for every order, as there might be sporadic extremal graphs or digraphs other than $Q_3$ to the two conjectures. 
	The question about finding the maximum total distance given $n$ and $r$ has been considered as well in the three scenarios. 
	The asymptotic results for these and related questions are summarized in Table~\ref{table1} and Table~\ref{table2}. For purpose of brevity and intuition, the results are summarized in terms of average rather than total distances.
	
	A problem which may be interesting to solve exactly is the following one.
	\begin{p}
		What is the maximum Wiener index of a biconnected digraph with out-radius $r>1$?
		Characterize the extremal digraphs.
	\end{p}
		It may be possible that the digraph in Figure~\ref{fig:digraph_rad+r} is the unique extremal digraph for this problem, corresponding to an the analog of Theorem~\ref{rad+=1}.

\begin{table}[h]
	
	\centering
	
	\begin{tabular}{ |l | r | r | }
		\hline 
		$\mu{(G)}$ & min & max\\
		\hline
		graphs &  $1+d^2\Theta\left( \frac{1}{n} \right)$  \cite{P84} & $d - d^{1.5} \Theta\left( \frac{1}{\sqrt{n}} \right) $ \cite{SC19}\\
		&  $1+r^2\Theta\left( \frac{1}{n} \right)$ Thm~\ref{main}& $2r - r^{1.5} \Theta\left( \frac{1}{\sqrt{n}} \right) $ Cor~\ref{cor_radius}\\
		trees & $2+d^2\Theta\left( \frac{1}n\right)$ \cite{LP} &  $\frac d2 + \lfloor \frac d2 \rfloor - d^{1.5} \Theta \left( \frac{1}{\sqrt{n}} \right)$ \cite{SC19}  \\
		& $2+r^2\Theta\left( \frac{1}n\right)$ \cite{LP} &  $2r - r^{1.5} \Theta\left( \frac{1}{\sqrt{n}} \right)$ Cor~\ref{cor_radius} \\
		\hline
	\end{tabular}
	\caption{minimum and maximum average distances for graphs}
	\label{table1}
\end{table}
\begin{table}[h]
	
	\centering
	\begin{tabular}{ |l | c | r | }
		\hline
		$\mu (G)$ for digraphs & min & max\\
		\hline
		$\rad^+$/$\rad^-$ &  $1+r^2\Theta \left( \frac{1}{n} \right)$ Thm~\ref{maindi} & $\frac n3 + r \Theta(1)$ Thm~\ref{maxmurad+}\\
		$\rad$ & $1+r^2\Theta\left( \frac{1}{n} \right)$ Thm~\ref{min_rad}&$2r-r^2 \Theta(\frac 1n)$ Thm~\ref{maxmurad}\\
		d & $1+d^2\Theta\left( \frac{1}{n} \right)$ \cite{P84} & $d-d^2\Theta(\frac{1}{n})$ \cite{SC19} \\
		\hline
	\end{tabular}
	\caption{minimum and maximum average distances for digraphs}
	\label{table2}
\end{table}

	We also obtained relationships between the extremal graphs and digraphs attaining the mimumum average distance/ maximum size given order and radius/ diameter. 

Perhaps the most basic question arising from this paper is conjecture~\ref{VZdi}, what is the maximum size of biconnected digraphs given order and outradius?

\bibliographystyle{abbrv}
\bibliography{MaxMuD}

\end{document}